\documentclass[onefignum,onetabnum]{siamart190516}

\usepackage{lipsum}
\usepackage{amsfonts}
\usepackage{graphicx}
\usepackage{epstopdf}
\usepackage{algorithmic}
\usepackage{amsmath}
\usepackage{amssymb}
\usepackage{bm}
\usepackage{diagbox}
\ifpdf
  \DeclareGraphicsExtensions{.eps,.pdf,.png,.jpg}
\else
  \DeclareGraphicsExtensions{.eps}
\fi
\usepackage{etoolbox}

\graphicspath{{img/}}

\newsiamremark{remark}{Remark}
\newsiamremark{hypothesis}{Hypothesis}
\newsiamremark{assumption}{Assumption}
\crefname{hypothesis}{Hypothesis}{Hypotheses}
\newsiamthm{claim}{Claim}

\headers{Generating function method}{M.J Gander, T. Lunet, D. Ruprecht and R. Speck}

\title{A unified analysis framework for iterative parallel-in-time algorithms
\thanks{Submitted to the editors DATE.
\funding{This project has received funding from the European High-Performance Computing Joint Undertaking (JU) under grant agreement No 955701. The JU receives support from the European Union's Horizon 2020 research and innovation programme and Belgium, France, Germany, and Switzerland. This project also received funding from the German Federal Ministry of Education and Research (BMBF) grant 16HPC048.}}}

\author{
Martin J. Gander%
\thanks{University of Geneva, Switzerland.}
\and
Thibaut Lunet\thanks{%
	Hamburg University of Technology, Germany 
	(\email{thibaut.lunet@tuhh.de}).}
\and
Daniel Ruprecht\footnotemark[3]
\and
Robert Speck\thanks{Forschungszentrum J\"ulich GmbH, Germany.}
}

\usepackage{amsopn}

\usepackage{xspace}
\usepackage{stmaryrd}

\newcommand{\Parareal}{\textsc{Para\-real}\xspace}
\newcommand{\PFASST}{\textsc{PFASST}\xspace}
\newcommand{\MGRIT}{\textsc{MGRIT}\xspace}
\newcommand{\RIDC}{\textsc{RIDC}\xspace}
\newcommand{\STMG}{\textsc{STMG}\xspace}

\newcommand{\G}{\mathcal{G}}
\newcommand{\F}{\mathcal{F}}

\newcommand{\norm}[1]{\left\lVert#1\right\rVert}
\newcommand{\matr}[1]{\mathbf{#1}}
\newcommand{\vect}[1]{\boldsymbol{#1}}

\newcommand{\dt}{\Delta t}

\newcommand{\Imat}{\matr{I}}
\newcommand{\Qmat}{\matr{Q}}
\newcommand{\QDelta}{\matr{Q}_\Delta}
\newcommand{\Hmat}{\matr{H}}

\newcommand{\uvect}{\vect{u}}

\newcommand{\MJac}{\matr{P}_{Jac}}
\newcommand{\MGS}{\matr{P}_{GS}}

\newcommand{\QTilde}{\matr{\tilde{Q}}}
\newcommand{\HTilde}{\matr{\tilde{H}}}

\newcommand{\MGSTilde}{\matr{\tilde{P}}_{GS}}
\newcommand{\QDeltaTilde}{\matr{\tilde{Q}}_\Delta}
\newcommand{\TFtoCBar}{\matr{\bar{T}}_F^C}
\newcommand{\TCtoFBar}{\matr{\bar{T}}_C^F}
\newcommand{\TFtoC}{\matr{T}_F^C}
\newcommand{\TCtoF}{\matr{T}_C^F}

\newcommand{\AMat}{\matr{A}}
\newcommand{\BMat}{\matr{B}}
\newcommand{\DMat}{\matr{D}}
\newcommand{\MMat}{\matr{M}}
\newcommand{\RMat}{\matr{R}}
\newcommand{\eyeMat}{\matr{I}}
\newcommand{\phiOp}{\bm{\phi}}
\newcommand{\chiOp}{\bm{\chi}}

\newcommand{\phiApprox}{\bm{\tilde{\phi}}}

\newcommand{\CoarseId}{C}
\newcommand{\MCoarse}{M^\CoarseId}
\newcommand{\uCoarse}{\vect{u}^\CoarseId}
\newcommand{\ACoarse}{\matr{A}_\CoarseId}
\newcommand{\chiCoarse}{\bm{\chi}_\CoarseId}
\newcommand{\phiCoarse}{\bm{\phi}_\CoarseId}

\newcommand{\phiApproxCoarse}{\bm{\tilde{\phi}}_\CoarseId}

\newcommand{\eg}{\textit{e.g.}~}
\newcommand{\ie}{\textit{i.e.}~}
\newcommand{\cf}{\textit{cf.}~}

\newlength\Colsep
\makeatletter
\patchcmd{\@addmarginpar}{\ifodd\c@page}{\ifodd\c@page\@tempcnta\m@ne}{}{}
\makeatother
\reversemarginpar

\ifpdf
\hypersetup{
  pdftitle={Generating Functions Method},
  pdfauthor={M.J. Gander, T. Lunet, D. Ruprecht and R. Speck}
}
\fi

\definecolor{tibo}{RGB}{139,0,128}

\newcommand{\ones}{\mbox{1\hspace{-0.24em}I}}

\newcommand{\TMG}{\textsc{TMG}\xspace}
\newcommand{\TMGCoarse}{\textsc{TMG$_c$}\xspace}
\newcommand{\TMGFine}{\textsc{TMG$_f$}\xspace}
\newcommand{\errBnd}{\theta}

\begin{document}

\maketitle

\begin{abstract}
  Parallel-in-time integration has been the focus of intensive
  research efforts over the past two decades due to the advent of
  massively parallel computer architectures and the scaling limits of
  purely spatial parallelization.  Various iterative parallel-in-time
  (PinT) algorithms have been proposed, like
  \Parareal, \PFASST, \MGRIT, and Space-Time Multi-Grid (\STMG).
  These methods have been described using different notations, and the
  convergence estimates that are available are
  difficult to compare.  We describe \Parareal, \PFASST, \MGRIT and
  \STMG for the Dahlquist model problem using a common notation and
  give precise convergence estimates using generating functions.
  This  allows us, for the first time, to directly compare their
  convergence.
  We prove that all four methods eventually converge
  super-linearly, and also compare them numerically.
  The generating function framework provides further
  opportunities to explore and analyze existing and new methods.
\end{abstract}

\begin{keywords}
  Parallel in Time (PinT) methods, \Parareal, \PFASST, \MGRIT,
  space-time multi-grid (\STMG), generating functions, convergence
  estimates.
\end{keywords}

\begin{AMS}
  65R20, 45L05, 65L20
\end{AMS}

\tableofcontents

\section{Introduction}
\label{sec:intro}

The efficient numerical solution of time-dependent ordinary and
partial differential equations (ODEs/PDEs) has always been an
important research subject in computational science and engineering.
Nowadays, with high-performance computing platforms providing more and more
processors whose individual processing speeds are no longer increasing, the
capacity of algorithms to run concurrently becomes important.
As classical parallelization algorithms start to reach their intrinsic
efficiency limits, substantial research efforts have been invested to find
new parallelization approaches that can translate the computing power of
modern many-core high-performance computing architectures into faster
simulations.

For time-dependent problems, the idea to parallelize across the time
direction has gained renewed attention in the last two decades\footnote{
	See also \url{https://www.parallel-in-time.org}}.
Various algorithms have been developed, for overviews see the papers by
Gander~\cite{gander2015years} or Ong and Schroder~\cite{OngEtAl2020}.
Four iterative algorithms have received significant attention, namely \Parareal~\cite{lions2001parareal} (474 citat.~since 2001)\footnote{
	Number of citations since publication, according to Google Scholar in April 2023.},
the \emph{Parallel Full Approximation Scheme in Space and Time}
(\PFASST)~\cite{emmett2012toward} (254 citat.~since 2012),
\emph{Multi-Grid Reduction In Time},
(\MGRIT)~\cite{friedhoff2013multigrid,falgout2014parallel}
(287~citat.~since 2014)
and a specific form of \emph{Space-Time Multi-Grid} (STMG)~\cite{gander2016analysis} (140~citat.~since 2016).
Other algorithms have been proposed, \eg the \emph{Parallel (or \Parareal) Implicit Time integration Algorithm} PITA~\cite{FarhatEtAl2003} (275~citat.~since 2003) which is very similar to \Parareal, the diagonalization technique~\cite{maday2008parallelization} (63~citat. since 2008), \emph{Revisionist Integral Deferred Corrections} (\RIDC)~\cite{christlieb2010parallel} (114~citat.~since 2010), \textsc{ParaExp}~\cite{gander2013paraexp} (103~citat.~since 2013) or \emph{parallel Rational approximation of EXponential Integrators} (REXI)~\cite{schreiber2018beyond} (28~citat.~since 2018).

\Parareal, \PFASST, \MGRIT and \STMG have all been benchmarked for large-scale problems using large numbers of cores of high-performance computing systems~\cite{hofer2019parallel,lecouvez2016parallel,lunet2018time,SpeckEtAl2012}.
They cast the solution process in time as a large linear or nonlinear system which is solved by iterating on all time steps simultaneously.
Since parallel performance is strongly linked to the rate of convergence, understanding convergence mechanisms and obtaining reliable error bounds for these iterative PinT methods is crucial.
Individual analyses exist for \Parareal~\cite{Bal2005,gander2008nonlinear,gander2007analysis,ruprecht2018wave,staff2005stability},
 \MGRIT~\cite{dobrev2016two,hessenthaler2020multilevel,southworth2019necessary},  \PFASST~\cite{bolten2017multigrid,bolten2018asymptotic}, and \STMG~\cite{gander2016analysis}.
There are also a few combined analyses showing equivalences between \Parareal and \MGRIT
\cite{falgout2014parallel,gander2018multigrid}
or connections between \MGRIT and \PFASST~\cite{minion2015interweaving}.
However, no systematic comparison of convergence behaviour, let alone
efficiencies, between these methods exists.

There are at least three obstacles to comparing these four methods: first, there is no common formalism or notation
to describe them; second, the existing analyses use very different techniques to
obtain convergence bounds; third, the
algorithms can be applied to many different problems in different ways with many tunable parameters, all of which
affect performance~\cite{GoetschelEtAl2021}.
Our main contribution is to address, at least for the Dahlquist test problem,
the first two problems by proposing a common formalism to rigorously describe
\Parareal, \PFASST, \MGRIT\footnote{
  We do not analyze in detail \MGRIT with FCF relaxation, only with F relaxation, in which case the two level variant is equivalent
    to \Parareal. Our framework could however be extended to include FCF relaxation, see Remark~\ref{rem:FCF}.}
and the Time Multi-Grid (\TMG) component\footnote{
    Since we focus only on the time dimension, the spatial component of \STMG is left out.
} of \STMG using the same notation.
Then, we obtain comparable error bounds for all four methods by
using the Generating Function Method (GFM)~\cite{knuth1975art}.
GFM has been used to analyze \Parareal~\cite{gander2008nonlinear} and was used to relate \Parareal and \MGRIT~\cite{gander2018multigrid}.
However, our use of GFM to derive common convergence bounds across multiple
algorithms is novel,
as is the presented unified framework.
When coupled with a predictive model for computational cost, this GFM framework could eventually be extended to a model to compare parallel performance of different algorithms, but this is left for future work.

Our manuscript is organized as follows: In
  Section~\ref{sec:gfmDescription}, we introduce the GFM framework.  In
  particular, in Section~\ref{sec:mainDef}, we give three definitions (time
  block, block variable and block operator) used to build the GFM
  framework and provide some examples using classical
  time integration methods.  Section~\ref{subsec:block_iter}
  contains the central definition of a \emph{block iteration} and
  again examples. In  Section~\ref{sec:gfmErrorBound}, we state the
  main theoretical results and error bounds, and the next sections
  contain how existing algorithms from the PinT literature can be
  expressed in the GFM framework: \Parareal in Section~\ref{sec:Parareal}, \TMG in Section~\ref{sec:STMG}, and \PFASST in
   Section~\ref{sec:PFASST}.  Finally,
  we compare in Section~\ref{sec:comparison} all methods using
  the GFM framework.  Conclusions and an outlook are given in
  Section~\ref{sec:conclusion}.
\section{The Generating Function Method}\label{sec:gfmDescription}

We consider the Dahlquist equation
\begin{equation}\label{eq:dahlquist}
\frac{du}{dt} = \lambda u,
\quad \lambda \in \mathbb{C},
\quad t \in (0, T],
\quad u(0) = u_0 \in \mathbb{C}.
\end{equation}
The complex parameter $\lambda$ allows us to emulate problems of parabolic ($\lambda<0$), hyperbolic ($\lambda$ imaginary) and mixed type.

\subsection{Blocks, block variables, and block operators}
\label{sec:mainDef}

We decompose the time interval $[0, T]$ into $N$ time sub-intervals
$[t_n, t_{n+1}]$ of uniform size $\dt$ with $n \in \{0, ..., N-1\}$.
\begin{definition}[time block]\label{def:blocks}
  A \emph{time block} (or simply block) denotes the
  discretization of a time sub-interval $[t_n, t_{n+1}]$ using
  $M>0$ grid points,
  \begin{equation}\label{eq:nodes}
    \tau_{n,m} = t_n + \dt \tau_{m}, \quad m\in\{1, ..., M\},
  \end{equation}
  where the $\tau_{m} \in [0,1]$ denote normalized grid points in
  time used for all blocks.
\end{definition}
We choose the name ``block'' in order to have a generic name for the
internal steps inside each time sub-interval.  A block could be
several time steps of a classical time-stepping scheme (\eg
Runge-Kutta, \cf Section~\ref{ex:RungeKutta}), the quadrature nodes of
a collocation method (\cf Section~\ref{ex:collocation}) or a
combination of both.
But in every case, a block represents the time domain that is associated
to one computational process of the time parallelization.
A block can also collapse by setting $M:=1$ and
$\tau_1:=1$, so that we retrieve a standard uniform
time-discretization with time step $\dt$.
The additional structure provided by blocks will be useful when
describing and analyzing two-level methods which use different numbers of grid
points per block for each level, \cf Section~\ref{sec:coarseGridCorrection}.

\begin{definition}[block variable]
  A {\em block variable} is a vector
  \begin{equation}
    \uvect_n = [u_{n,1},u_{n,2},\ldots,u_{n,M}]^T,
  \end{equation}
  where $u_{n,m}$ is an approximation of $u(\tau_{n,m})$ on
  the time block for the time sub-interval $[t_n, t_{n+1}]$. For
  $M=1$, $\uvect_n$ reduces to a scalar approximation of
  $u(\tau_{n,M})\equiv u(t_{n+1})$.
\end{definition}
Some iterative PinT methods like \Parareal (see Section~\ref{sec:Parareal}) use values defined at the
interfaces between sub-intervals $[t_n, t_{n+1}]$.
Other algorithms, like \PFASST (see Section~\ref{sec:PFASST}),
update solution values in the interior of blocks.
In the first case, the block variable is the right interface value with $M=1$ and thus $\tau_{1}=1$.
In the second case, it consists of \emph{volume} values in the time block $[t_n, t_{n+1}]$ with $M>1$.
In both cases, PinT algorithms can be defined as \emph{iterative processes updating the block variables.}

\begin{remark}
  While the adjective ``time'' is natural for evolution problems, PinT
  algorithms can also be applied to recurrence relations in different contexts
  like deep learning~\cite{gunther2020layer} or
  when computing Gauss quadrature formulas~ \cite{gander2021parastieltjes}.
  Therefore, we will not systematically mention ``time'' when talking about
  blocks and block variables.
\end{remark}

\begin{definition}[block operators]
  \label{def:blockOperators}
	We denote as \emph{block operators} the two linear functions
  $\phiOp:\mathbb{C}^M\rightarrow\mathbb{C}^M$ and
  $\chiOp:\mathbb{C}^M\rightarrow\mathbb{C}^M$ for which the block
  variables of a numerical solution of \eqref{eq:dahlquist} satisfy
  \begin{equation}\label{eq:blockProblem}
    \phiOp(\uvect_1) = \chiOp(u_0\vect{\ones}),\quad
    \phiOp(\uvect_{n+1})= \chiOp(\uvect_n),\quad n=1,2,\ldots,N-1,
  \end{equation}
  with $\vect{\ones}:=[1,\dots,1]^T$.  The \emph{time integration operator}
  $\phiOp$ is bijective and $\chiOp$ is a \emph{transmission
    operator}.  The {\em time propagator} updating $\uvect_n$ to
  $\uvect_{n+1}$ is given by
  \begin{equation}\label{eq:blockSeqProp}
	\bm{\psi} := \phiOp^{-1}\chiOp.
  \end{equation}
\end{definition}
\subsubsection{Example with Runge-Kutta methods}\label{ex:RungeKutta}

Consider numerical integration of~\eqref{eq:dahlquist} with a
Runge-Kutta method with stability function
\begin{equation}
  R(z)\approx e^z.
\end{equation}
Using $\ell$ equidistant time steps per block, there are two natural
ways to write the method using block operators:
\begin{enumerate}
  \item {\em The volume formulation}: set $M:=\ell$ with $\tau_{m}:=m/M$, $m=1,\ldots,M$.
  	Setting $r:=R(\lambda\dt/\ell)^{-1}$, the block operators are the
  	$M\times M$ sparse matrices
  \begin{equation}
	\phiOp := \begin{pmatrix}
	r & \\
	-1 & 	r &\\
   & \ddots & \ddots
	\end{pmatrix},\quad
	\chiOp := \begin{pmatrix}
	0 & \dots & 0 & 1\\
	\vdots &  & \vdots & 0\\
	\vdots &  & \vdots & \vdots
	\end{pmatrix}.
	\end{equation}
	\item {\em The interface formulation}: set $M:=1$ so that
	\begin{equation}\label{eq:rungeKuttaU}
	\phiOp:=R(\lambda\dt/\ell)^{-\ell},\quad \chiOp:=1.
	\end{equation}
\end{enumerate}

\subsubsection{Example with collocation methods}\label{ex:collocation}

Collocation methods are special implicit Runge-Kutta
methods~\cite[Chap.~IV, Sec.~4]{wanner1996solving} and instrumental
when defining \PFASST in Section~\ref{sec:PFASST}.  We show their
representation with block operators.  Starting from the Picard
formulation for \eqref{eq:dahlquist} in one time sub-interval
  $[t_n, t_{n+1}]$,
\begin{equation}
  u(t) = u(t_n) + \int_{t_n}^{t} \lambda u(\tau)d\tau,
\end{equation}
we choose a quadrature rule to approximate the integral.
We consider only Lobatto or Radau-II type quadrature nodes where the last quadrature node coincides with the right sub-interval boundary.
This gives us quadrature nodes for each sub-interval that form the block discretization points $\tau_{n,m}$ of Definition~\ref{def:blocks}, with $\tau_M=1$.
We approximate the solution $u(\tau_{n,m})$ at each node by
\begin{equation}
  u_{n,m} = u_{n,0} + \lambda \dt \sum_{j=1}^{M} q_{m,j} u_{n,j}
	\quad\text{with}\quad
	q_{m,j} := \int_{0}^{\tau_m} l_j(s)ds,
\end{equation}
where $l_j$ are the Lagrange polynomials associated with the nodes
$\tau_{m}$.  Combining all the nodal values, we form the block
variable $\uvect_{n}$, which satisfies the linear system
\begin{equation}\label{eq:collocation}
	(\Imat-\Qmat) \uvect_{n} = \begin{pmatrix}
		u_{n,0} \\ \vdots \\ u_{n,0}
	\end{pmatrix} =
	\begin{bmatrix}
		0 & \dots & 0 & 1 \\
		\vdots & & \vdots & \vdots \\
		0 & \dots & 0 & 1
	\end{bmatrix}\uvect_{n-1} =: \Hmat \uvect_{n-1},
\end{equation}
with the quadrature matrix $\Qmat := \lambda\dt (q_{m,j})$, $\Imat$
the identity matrix, and $\Hmat$ sometimes called the transfer
matrix that copies the last value of the previous time block to obtain the initial value $u_{n,0}$ of the current block\footnote{This specific form of the matrix $\Hmat$ comes from the use of Lobatto or Radau-II rules, which treat the right
  interface of the time sub-interval as a node.  A similar
  description can also be obtained for Radau-I or Gauss-type quadrature rules that do
  not use the right boundary as node, but we omit it for the sake of simplicity.}.
The integration and transfer block operators from
Definition~\ref{def:blockOperators} then become\footnote{
	The notation $\Hmat$ is specific to SDC and collocation methods (see \eg \cite{bolten2017multigrid}), while the $\chiOp$ notation from the GFM framework is generic for arbitrary time integration methods.}
$\phiOp := (\Imat-\Qmat)$, $\chiOp := \Hmat$.

\subsection{Block iteration}\label{subsec:block_iter}

Having defined the block operators for our problem,
we write the numerical approximation~\eqref{eq:blockProblem} of~\eqref{eq:dahlquist} as the {\em all-at-once
global problem}
\begin{equation}\label{eq:globalProblem}
\AMat\uvect :=
\begin{pmatrix}
\phiOp & & &\\
-\chiOp & \phiOp & &\\
& \ddots & \ddots &\\
& & -\chiOp & \phiOp
\end{pmatrix}
\begin{bmatrix}
\uvect_1\\\uvect_2\\\vdots\\\uvect_N
\end{bmatrix}
=
\begin{bmatrix}
\chiOp(u_0\vect{\ones})\\0\\\vdots\\0
\end{bmatrix}
=:
 \vect{f}.
\end{equation}
Iterative PinT algorithms solve~\eqref{eq:globalProblem} by updating a vector $\uvect^k =
[\uvect^k_1, \dots, \uvect^k_N]^T$ to $\uvect^{k+1}$ until some
stopping criterion is satisfied.
If the global iteration can be written as a local update for
each block variable separately, we call the
local update formula a \emph{block iteration}.
\begin{definition}[Primary block iteration]\label{def:primaryBlockIteration}
  A {\em primary block iteration} is an updating formula for $n\geq 0$
  of the form
  \begin{equation}\label{eq:primaryBlockIteration}
	\uvect^{k+1}_{n+1} = \BMat^0_1(\uvect^k_{n+1})
		+ \BMat^1_0\left(\uvect^{k+1}_{n}\right) + \BMat^0_0\left(\uvect^{k}_{n}\right),
        \quad\uvect_{0}^k = u_0\vect{\ones} \quad\forall k \in \mathbb{N},
  \end{equation}
  where $\BMat^0_1$, $\BMat^1_0$ and $\BMat^0_0$ are linear
  operators from $\mathbb{C}^M$ to $\mathbb{C}^M$ that satisfy the
  \emph{consistency condition}\footnote{\label{rem:consistency}
  Condition \eqref{eq:primaryBlockIterCondition}
    is necessary for the block iteration to have
      the correct fixed point.}
  \begin{equation}\label{eq:primaryBlockIterCondition}
	(\BMat_1^0-\eyeMat)\bm{\psi} + \BMat^1_0 + \BMat^0_0 = 0,
  \end{equation}
  with $\bm{\psi}$ defined in~\eqref{eq:blockSeqProp}.
\end{definition}
Note that a block iteration is always associated with an all-at-once global
problem, and the primary block iteration~\eqref{eq:primaryBlockIteration} should converge to the solution of~\eqref{eq:globalProblem}.

Figure~\ref{fig:kn-base} (left)
shows a graphical representation of a primary block iteration using a
\emph{$kn$-graph} to represent the dependencies of
$\uvect^{k+1}_{n+1}$ on the other block variables.  The $x$-axis
represents the block index $n$ (time), and the $y$-axis represents the
iteration index $k$.
Arrows show dependencies from previous $n$ or
$k$ indices and can only go from left to right and/or from bottom to
top.
For the primary block iteration, we consider only
dependencies from the previous block $n$ and iterate $k$ for
$\uvect_{n+1}^{k+1}$.
\begin{figure}
    \centering
    \includegraphics[height=0.2\linewidth]{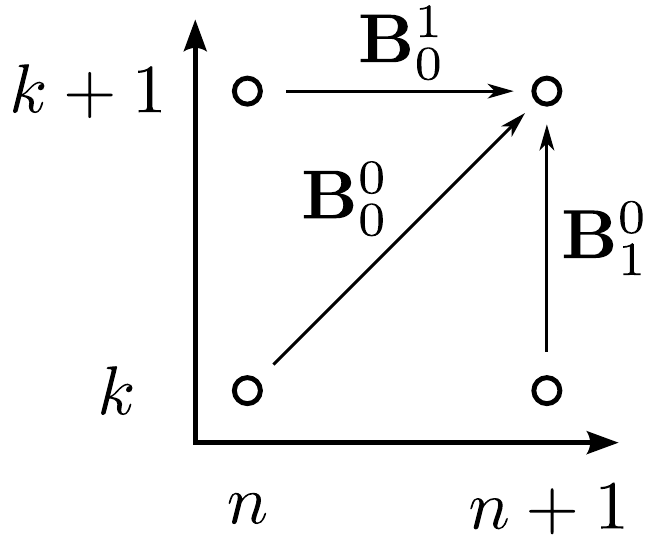}\hfil%
    \includegraphics[height=0.2\linewidth]{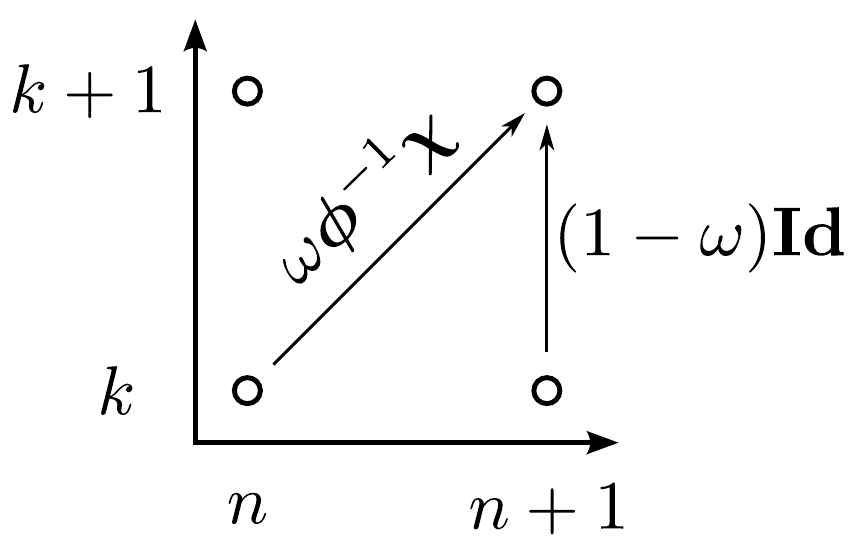}\hfil%
    \includegraphics[height=0.2\linewidth]{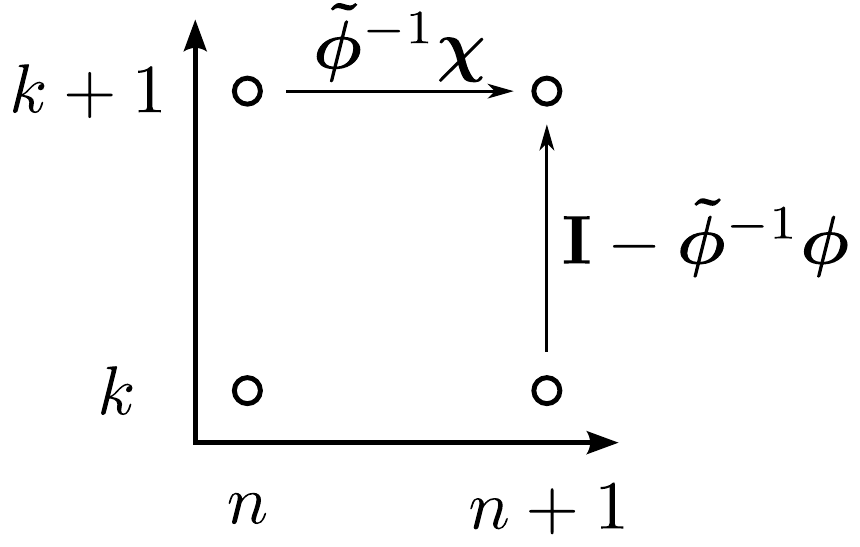}%
    \caption{$kn$-graphs for a generic Primary Block Iteration (left), damped Block Jacobi (middle) and Approximate Block Gauss-Seidel (right).}
    \label{fig:kn-base}
\end{figure}

More general block iterations can also be
considered for specific iterative PinT methods, \eg \MGRIT with
FCF-relaxation (see Remark~\ref{rem:FCF}).
Other algorithms also consist of combinations of two or more
block iterations, for example \STMG
(\cf Section~\ref{sec:STMG}) or \PFASST (\cf Section~\ref{sec:PFASST}).
But we show in those sections that we can reduce those combinations into a single primary block iteration,
hence we focus here mostly on primary block iterations to introduce our analysis framework.

We next describe the Block Jacobi relaxation (Section~\ref{ex:blockJacobi}) and the Approximate Block Gauss-Seidel iteration (Section~\ref{ex:ABGS}), which are key components used to describe iterative PinT methods.

\subsubsection{Block Jacobi relaxation}\label{ex:blockJacobi}

A damped block Jacobi iteration for the global problem~\eqref{eq:globalProblem} can be written as
\begin{equation}\label{eq:jacobiRelaxationGlobal}
    \uvect^{k+1} = \uvect^k + \omega\DMat^{-1}(\vect{f}-\AMat\uvect^k),
\end{equation}
where $\DMat$ is a block diagonal matrix constructed with the integration operator $\phiOp$, and $\omega>0$ is a
relaxation parameter.
For $n>0$, the corresponding block formulation is
\begin{equation}\label{eq:jacobiRelaxation}
    \uvect_{n+1}^{k+1} = (1-\omega)\uvect_{n+1}^k + \omega\phiOp^{-1}\chiOp\uvect_{n}^k,
\end{equation}
which is a primary block iteration with $\BMat_0^1 = 0$.
Its $kn$-graph is shown in Figure~\ref{fig:kn-base} (middle).
The consistency condition~\eqref{eq:primaryBlockIterCondition} is satisfied, since
\begin{equation}
    ((1-\omega)\eyeMat-\eyeMat)\phiOp^{-1}\chiOp + 0 + \omega\phiOp^{-1}\chiOp = 0.
\end{equation}
Note that selecting $\omega=1$ simplifies the block iteration to
\begin{equation}\label{eq:nonDampedBJ}
    \uvect_{n+1}^{k+1} = \phiOp^{-1}\chiOp\uvect_{n}^k.
\end{equation}

\subsubsection{Approximate Block Gauss-Seidel iteration}
\label{ex:ABGS}

Let us consider a Block Gauss-Seidel type preconditioned iteration
for the global problem~\eqref{eq:globalProblem},
\begin{equation}\label{eq:ABGSGlobal}
    \uvect^{k+1} = \uvect^k + \MGS^{-1}(\vect{f}-\AMat\uvect^k),\quad
    \MGS =
    \begin{bmatrix}
        \phiApprox & & \\
        - \chiOp & \phiApprox & \\
        & \ddots & \ddots
    \end{bmatrix},
\end{equation}
where the block operator $\phiApprox$ corresponds to an approximation of $\phiOp$.
This approximation can be based on time-step coarsening,
but could also use other approaches,
\eg a lower-order time integration method.
In general, $\phiApprox$ must be cheaper than $\phiOp$, but is also less accurate.
Subtracting $\uvect^k$ in \eqref{eq:ABGSGlobal} and multiplying by $\MGS$ yields
the block iteration of this \emph{Approximate Block Gauss-Seidel} (ABGS),
\begin{equation}
    \label{eq:ABGS}
    \vect{u}_{n+1}^{k+1} = \left[\eyeMat - \phiApprox^{-1} \phiOp\right] \vect{u}_{n+1}^{k} +
    \phiApprox^{-1}\chiOp\vect{u}_{n}^{k+1}.
\end{equation}
Its $kn$-graph is shown in Figure \ref{fig:kn-base} (right).
Note that a standard block Gauss-Seidel iteration
for~\eqref{eq:globalProblem} (\ie with $\phiApprox = \phiOp$)
is actually a direct solver,
the iteration converges in one
step by integrating all blocks with $\phiOp$ sequentially,
and its block iteration is simply
\begin{equation}
    \label{eq:exactBGS}
    \vect{u}_{n+1}^{k+1} =
    \phiOp^{-1}\chiOp\vect{u}_{n}^{k+1}.
\end{equation}
\subsection{Generating function and error bound for a block iteration}
\label{sec:gfmErrorBound}

Before giving a generic expression for the error bound of
the primary block iteration~\eqref{eq:primaryBlockIteration} using the GFM framework, we first need a definition
and a preliminary result.
The primary block iteration~\eqref{eq:primaryBlockIteration} is defined  for each block index $n \geq 0$, thus we can define
\begin{definition}[Generating function]\label{def:genfunc}
  The \emph{generating function} associated
  with the primary block iteration \eqref{eq:primaryBlockIteration}
  is the power series
  \begin{equation}
    \rho_k(\zeta) := \sum_{n=0}^{\infty} e^k_{n+1}\zeta^{n+1},
  \end{equation}
  where
  $
    e^k_{n+1} := \norm{\uvect^k_{n+1}-\uvect_{n+1}}
  $
  is the difference between the $k^{th}$ iterate $\uvect^k_{n+1}$ and the exact
  solution $\uvect_{n+1}$ for one block of~\eqref{eq:blockProblem} in some norm on
  $\mathbb{C}^M$.
\end{definition}
Since the analysis works in any norm, we do not specify a particular
one here.
In the numerical examples we use the $L^{\infty}$ norm on $\mathbb{C}^M$.

\begin{lemma}\label{lem:primary}
  The generating function for the primary block iteration~\eqref{eq:primaryBlockIteration} satisfies
  \begin{equation}
    \rho_{k+1}(\zeta) \leq
	\frac{\gamma +  \alpha \zeta}{1 - \beta \zeta}\rho_{k}(\zeta),
  \end{equation}
  where $\alpha := \norm{\BMat_0^0}$, $ \beta := \norm{\BMat_0^1}$, $
  \gamma := \norm{\BMat_1^0}$, and the operator norm is induced by the chosen vector norm.
\end{lemma}
\begin{proof}
  We start from~\eqref{eq:primaryBlockIteration} and subtract the exact solution of~\eqref{eq:blockProblem},
  \begin{equation}
    \uvect^{k+1}_{n+1} - \uvect_{n+1} = \BMat^0_1(\uvect^k_{n+1})
    + \BMat^1_0\left(\uvect^{k+1}_{n}\right) + \BMat^0_0\left(\uvect^{k}_{n}\right)
    - \bm{\psi}(\uvect_{n}).
  \end{equation}
  Using the linearity of the block operators and \eqref{eq:primaryBlockIterCondition} with $\uvect_{n}$, this simplifies to
  \begin{equation}\label{eq:errRecurrence}
    \uvect^{k+1}_{n+1} - \uvect_{n+1} = \BMat^0_1(\uvect^k_{n+1}-\uvect_{n+1})
	+ \BMat^1_0\left(\uvect^{k+1}_{n}-\uvect_{n}\right)
	+ \BMat^0_0\left(\uvect^{k}_{n}-\uvect_{n}\right).
  \end{equation}
  We apply the norm, use the triangle inequality and the operator norms defined above to get the recurrence relation
  \begin{equation}\label{eq:errRecurrenceScalar}
	e^{k+1}_{n+1} \leq \gamma e^{k}_{n+1} + \beta e^{k+1}_n + \alpha e^{k}_n
  \end{equation}
  for the error. We multiply this inequality by $\zeta^{n+1}$ and sum for $n\in \mathbb{N}$
  to get
  \begin{equation}
    \sum_{n=0}^{\infty}e^{k+1}_{n+1} \zeta^{n+1} \leq
	\gamma\sum_{n=0}^{\infty}e^{k}_{n+1} \zeta^{n+1} +
	\beta\sum_{n=0}^{\infty}e^{k+1}_{n} \zeta^{n+1} +
 	\alpha\sum_{n=0}^{\infty}e^{k}_{n} \zeta^{n+1}.
  \end{equation}
  Note that this is a formal power series expansion for $\zeta$ small in the sense of generating functions~\cite[Section~1.2.9]{knuth1975art}.
  Using Definition~\ref{def:genfunc} and that $e^k_0 = 0$ for
  all $k$ we find
  \begin{equation}
	\rho_{k+1}(\zeta) \leq \gamma\rho_{k}(\zeta) +
        \beta\zeta\sum_{n=1}^{\infty}e^{k+1}_{n} \zeta^{n} +
        \alpha\zeta\sum_{n=1}^{\infty}e^{k}_{n} \zeta^{n}.
  \end{equation}
  Shifting indices leads to
  \begin{equation}
    (1-\beta\zeta)\rho_{k+1}(\zeta) \leq (\gamma + \alpha\zeta)\rho_{k}(\zeta)
  \end{equation}
  and concludes the proof.
\end{proof}
\begin{theorem}\label{th:errBoundPBI}
  Consider the primary block
  iteration~\eqref{eq:primaryBlockIteration} and let
  \begin{equation}\label{eq:deltaDefinition}
    \delta := \underset{n = 1, \ldots, N}{\max}\norm{\uvect^0_n-\uvect_n}
  \end{equation}
  be the maximum error of the initial guess over all blocks.
  Then, using the notation of Lemma~\ref{lem:primary}, we have
  \begin{equation}
    e^k_{n+1} \leq \errBnd_{n+1}^k(\alpha, \beta, \gamma)\delta
  \end{equation}
  for $k>0$, where $\errBnd_{n+1}^k$ is a bounding function defined as follows:
  \begin{itemize}
    \item if {\em only} $\gamma=0$, then\\
	\noindent\begin{minipage}{0.9\textwidth}
	\begin{minipage}[c][2cm][c]{\dimexpr0.7\textwidth-0.5\Colsep\relax}
	\begin{equation}\label{Bound1}
  	  \errBnd_{n+1}^k = \frac{\alpha^k}{(k-1)!}
	  \sum_{i=0}^{n-k}\prod_{l=1}^{k-1}(i+l)\beta^{i};
	\end{equation}
	\end{minipage}%
	\begin{minipage}[c][2cm][c]{\dimexpr0.3\textwidth-0.5\Colsep\relax}
	  \centering\includegraphics[height=2cm]{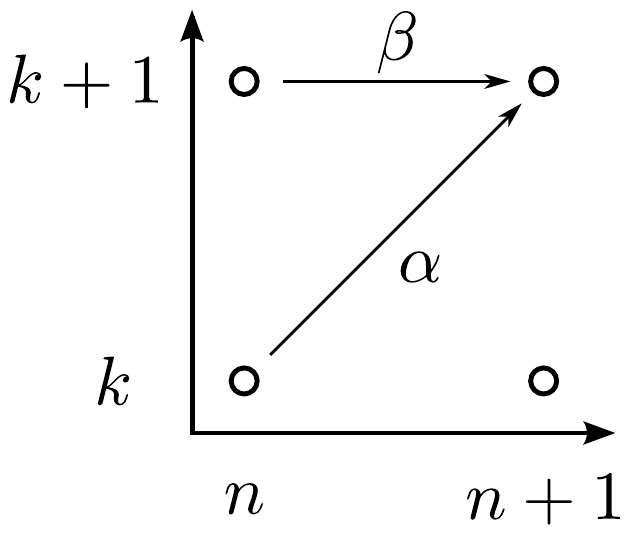}\hfil%
	\end{minipage}%
	\end{minipage}\\
    \item if {\em only} $\beta=0$, then\\
	\noindent\begin{minipage}{0.9\textwidth}
	\begin{minipage}[c][2cm][c]{\dimexpr0.7\textwidth-0.5\Colsep\relax}
  	  \begin{equation}
	    \errBnd_{n+1}^k = \begin{cases}
	    (\gamma + \alpha)^k \text{ if } k \leq n, \\
	    \displaystyle \gamma^k \sum_{i=0}^{n}
	    \binom{k}{i}\left(\frac{\alpha}{\gamma}\right)^i \text{ otherwise;}
	    \end{cases}
	  \end{equation}
	\end{minipage}%
	\begin{minipage}[c][2cm][c]{\dimexpr0.3\textwidth-0.5\Colsep\relax}
	  \centering\includegraphics[height=2cm]{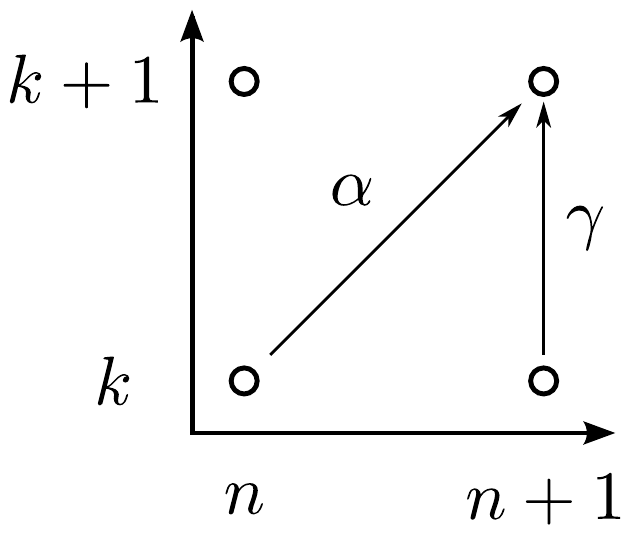}\hfil%
	\end{minipage}%
	\end{minipage}\\
    \item if {\em only} $\alpha=0$, then\\
	\noindent\begin{minipage}{0.9\textwidth}
	\begin{minipage}[c][2cm][c]{\dimexpr0.7\textwidth-0.5\Colsep\relax}
	\begin{equation}
  	  \errBnd_{n+1}^k =
  	  \frac{\gamma^k}{(k-1)!} \sum_{i=0}^{n}\prod_{l=1}^{k-1}(i+l)\beta^{i};
	\end{equation}
	\end{minipage}%
	\begin{minipage}[c][2cm][c]{\dimexpr0.3\textwidth-0.5\Colsep\relax}
	  \centering\includegraphics[height=2cm]{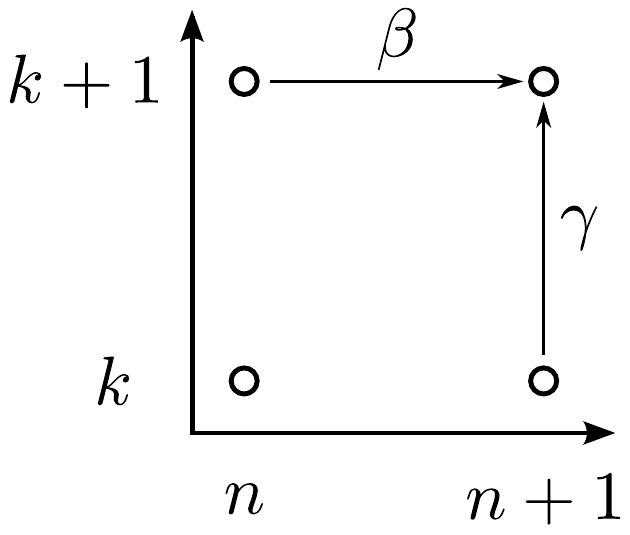}\hfil%
	\end{minipage}%
	\end{minipage}\\
    \item if {\em neither} $\alpha$, nor $\beta$, nor $\gamma$ are zero, then
	\begin{equation}\label{Bound4}
	\errBnd_{n+1}^k = \gamma^k \sum_{i=0}^{\min(n, k)} \sum_{l=0}^{n-i}
	\binom{k}{i}\binom{l+k-1}{l}
	\left(\frac{\alpha}{\gamma}\right)^i\beta^l.
	\end{equation}
  \end{itemize}
  We call any error bound obtained from one of these formulas a \emph{GFM-bound}.
\end{theorem}
The proof uses Lemma~\ref{lem:primary} to bound the generating function at $k=0$ by
\begin{equation}
\rho_0(\zeta) \leq \delta \sum_{n=0}^{\infty} \zeta^{n+1},
\end{equation}
which covers arbitrary initial guesses for defining starting
  values $\uvect_n^0$ for each block.
For specific initial guesses, $\rho_{0}(\zeta)$ can be bounded differently~\cite[Proof of Th.~1]{gander2008nonlinear}.
The error bound is then computed by coefficient identification after a power series expansion.
The rather technical proof can be found in Appendix~\ref{ap:errPrimary}.

In the numerical examples shown below, we find that the estimate
from Theorem~\ref{th:errBoundPBI} is not always sharp, \cf Section~\ref{sec:analysisBlockSDC}.
If the last time point of the blocks coincides with the right bound of the
sub-interval\footnote{
This is the case for all time-integration methods considered in this paper, even if this is not a necessary condition to use the GFM framework.},
it is helpful to define the \emph{interface error} at
the right boundary point of the $n^{th}$ block as
\begin{equation}
  \bar{e}_{n+1}^k := |\bar{u}^k_{n+1}-\bar{u}_{n+1}|,
\end{equation}
where $\bar{u}$ is the last element of the block variable $\uvect$.
We then multiply \eqref{eq:errRecurrence} by $e_M^T=[0,\dots,0,1]$
to get
\begin{equation}
  e_M^T (\uvect_{n+1}^{k+1}-\uvect_{n+1})
= \vect{b}_1^0 (\uvect_{n+1}^{k}-\uvect_{n+1}) +
\vect{b}_0^1 (\uvect_{n}^{k+1}-\uvect_{n}) +
\vect{b}_0^0 (\uvect_{n}^{k}-\uvect_{n}),
\end{equation}
where $\vect{b}_i^j$ is the last row of the block operator
$\BMat_i^j$.  Taking the absolute value on both sides, we
recognize the interface error $\bar{e}_{n+1}^{k+1}$ on the left
hand side. By neglecting the error from interior points and using the
triangle inequality, we get the approximation\footnote{
	For an interface block iteration
	($M=1, \tau_{1}=1$), \eqref{InterfaceApproximateRelation} becomes a rigorous inequality and Corollary~\ref{cor:interfaceApproximation} thus becomes an upper bound.}
\begin{equation}\label{InterfaceApproximateRelation}
  \bar{e}_{n+1}^{k+1} \lessapprox \bar{\gamma}\bar{e}_{n+1}^{k}
+ \bar{\beta}\bar{e}_{n}^{k+1}
+ \bar{\alpha}\bar{e}_{n}^{k},
\end{equation}
where $\bar{\alpha}:=|\bar{b}_0^0|$, $\bar{\beta}:=|\bar{b}_1^0|$,
$\bar{\gamma}:=|\bar{b}_0^1|$.
\begin{corollary}[Interface error approximation]
	\label{cor:interfaceApproximation}
  Defining for the initial interface error the bound
    $\bar{\delta}:=\max_{n\in\{1,\dots, N\}} \norm{\bar{u}^0_n-
      \bar{u}_n}$, we obtain for the interface error the approximation
  \begin{equation}
    \bar{e}_{n+1}^{k} \lessapprox
    \bar{\errBnd}_{n+1}^k \bar{\delta},\quad
    \bar{\errBnd}_{n+1}^k:=\errBnd_{n+1}^k(\bar{\alpha}, \bar{\beta},\bar{\gamma}),
  \end{equation}
  with $\errBnd_{n+1}^k$ defined in Theorem~\ref{th:errBoundPBI}.
\end{corollary}
\begin{proof}
  The result follows as in the proof of Lemma~\ref{lem:primary}
    using approximate relations.
\end{proof}%
\begin{remark}\label{rem:interfaceVSblock}
    For the general case, the error at the interface $\bar{e}_{n+1}^{k+1}$ \emph{is not the same as}
    the error for the whole block $e_{n+1}^{k+1}$ .
    Only a block discretization using a single point ($M=1$) makes the two values identical.
    Furthermore, Corollary~\ref{cor:interfaceApproximation} is generally not an upper bound, but an approximation thereof.

\end{remark}

\section{Writing \Parareal and \MGRIT as block iterations}\label{sec:Parareal}

\subsection{Description of the algorithm}\label{sec:descrParareal}

 The \Parareal algorithm introduced by Lions et
  al.~\cite{lions2001parareal} corresponds to a block iteration update
  with scalar blocks ($M=1$), and its convergence was analyzed in
  \cite{gander2007analysis,ruprecht2018wave}.  We propose here a new
  description of \Parareal in the scope of the GFM framework, which
  states that \Parareal is simply a combination of two preconditioned
  iterations applied to the global problem \eqref{eq:globalProblem},
  namely one Block Jacobi relaxation without damping
  (Section~\ref{ex:blockJacobi}),
  followed by an ABGS iteration (Section~\ref{ex:ABGS}).

We denote by $\uvect^{k+1/2}$ the intermediate solution after the Block Jacobi step.
Using \eqref{eq:nonDampedBJ} and \eqref{eq:ABGS}, the two successive
primary block iteration steps are
\begin{gather}
    \uvect_{n+1}^{k+1/2} = \phiOp^{-1}\chiOp\uvect_{n}^k, \\
    \uvect_{n+1}^{k+1} = \left[\eyeMat - \phiApprox^{-1} \phiOp\right] \vect{u}_{n+1}^{k+1/2} +
    \phiApprox^{-1}\chiOp\vect{u}_{n}^{k+1}.
\end{gather}
Combining both yields the primary block iteration
\begin{equation}\label{eq:pararealBlock}
    \uvect_{n+1}^{k+1} = \left[\phiOp^{-1}\chiOp - \phiApprox^{-1}\chiOp\right]\uvect_{n}^k
    + \phiApprox^{-1}\chiOp\vect{u}_{n}^{k+1}.
\end{equation}
Now as stated in Section~\ref{ex:ABGS}, $\phiApprox$ is an
approximation of the integration operator $\phiOp$, that is
cheaper to invert but less accurate\footnote{ In the original
    paper \cite{lions2001parareal},
    this approximation is done using larger time-steps, but
    many other types of approximations have been used since then in the
    literature.}.  In other words, if we define
\begin{equation}
    \F := \phiOp^{-1}\chiOp, \quad \G := \phiApprox^{-1}\chiOp,
\end{equation}
to be a fine and coarse propagator on one block, then \eqref{eq:pararealBlock}
becomes
\begin{equation}\label{eq:pararealClassic}
    \uvect_{n+1}^{k+1} =
    \F \uvect_{n}^k + \G \uvect_{n}^{k+1} - \G \uvect_{n}^k,
\end{equation}
which is the \Parareal update formula derived from the approximate Newton update in the multiple shooting approximation in \cite{gander2007analysis}.
Iteration \eqref{eq:pararealClassic} is a primary block iteration in the sense of
Definition~\ref{def:primaryBlockIteration} with $\BMat^0_1:=0$,
$\BMat^1_0:=\mathcal{G}$ and $\BMat^0_0:=\mathcal{F}-\mathcal{G}$.
\begin{figure}
    \centering
    \includegraphics[height=0.2\linewidth]{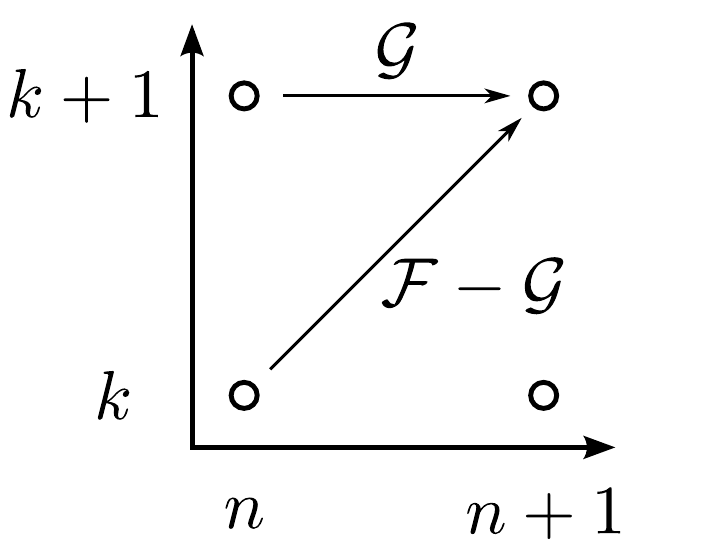}\hfil%
    \includegraphics[height=0.2\linewidth]{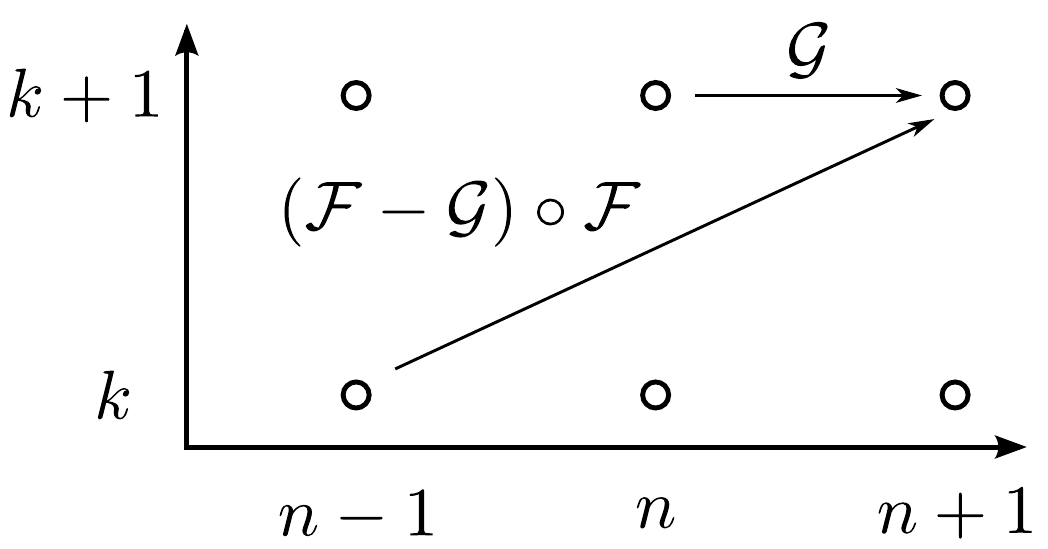}\hfil%
    \caption{$kn$-graphs for \Parareal/MGRIT with F-relaxation (left) and \MGRIT with FCF-relaxation/\Parareal with overlap (right).}
    \label{fig:kn-parareal}
\end{figure}
Its $kn$-graph is shown in Figure~\ref{fig:kn-parareal} (left).
The consistency condition~\eqref{eq:primaryBlockIterCondition} is satisfied, since
$
(0 - \eyeMat)\mathcal{F} + \mathcal{G} + (\mathcal{F}-\mathcal{G}) = 0.
$
If we subtract $\uvect_{n+1}^k$ in~\eqref{eq:pararealBlock}, multiply both sides by $\phiOp$ and rearrange terms, we can write \Parareal as the preconditioned fixed point iteration
\begin{equation}\label{eq:pararealGlobal}
    \uvect^{k+1} = \uvect^k + \MMat^{-1}(\vect{f}-\AMat\uvect^k),\quad
    \MMat:= \begin{pmatrix}
        \phiOp& &\\
        -\phiOp\phiApprox^{-1}\chiOp & \phiOp &\\
        & \ddots & \ddots
    \end{pmatrix},
\end{equation}
with iteration matrix $\RMat_{\Parareal} = \eyeMat - \MMat^{-1}\AMat$.

\begin{remark}\label{rem:FCF}
It is known in the literature that \Parareal is equivalent to a
  two-level \MGRIT algorithm with F-relaxation
  \cite{falgout2014parallel,gander2018multigrid,southworth2019necessary}.
  In \MGRIT, one however also often uses FCF-relaxation, which is a
  combination of \emph{two} non-damped ($\omega=1$) Block Jacobi
  relaxation steps, followed by an ABGS step: denoting by
  $\uvect^{k+1/3}$ and $\uvect^{k+2/3}$ the intermediary block Jacobi
  iterations, we obtain
\begin{gather}
    \uvect_{n+1}^{k+1/3} = \phiOp^{-1}\chiOp\uvect_{n}^k, \\
    \uvect_{n+1}^{k+2/3} = \phiOp^{-1}\chiOp\uvect_{n}^{k+1/3}, \\
    \uvect_{n+1}^{k+1} = \left[\eyeMat - \phiApprox^{-1} \phiOp\right] \vect{u}_{n+1}^{k+2/3} +
    \phiApprox^{-1}\chiOp\vect{u}_{n}^{k+1}.
\end{gather}
Shifting the $n$ index in the first Block Jacobi iteration, combining
all of them and re-using the $\mathcal{F}$ and $\mathcal{G}$ notation
then gives
\begin{equation}\label{eq:blockIterPararealOverlap}
    \uvect^{k+1}_{n+1} =
        \BMat^0_{-1}(\uvect^k_{n-1})
        + \BMat^1_0\left(\uvect^{k+1}_{n}\right),
        \quad \BMat^0_{-1} = (\mathcal{F} - \mathcal{G})\mathcal{F},\;
        \BMat^1_{0} = \mathcal{G},
\end{equation}
which is the update formula of \Parareal with overlap,
shown to be equivalent to \MGRIT with FCF-relaxation \cite[Th.~4]{gander2018multigrid}\footnote{
    It was shown in \cite{gander2018multigrid} that \MGRIT with (FC)$^{\nu}$F-relaxation,
    where $\nu > 0$ is the number of additional FC-relaxations, is equivalent to an overlapping version of \Parareal
    with $\nu$ overlaps.
    Generalizing our computations shows that those algorithms are equivalent to $(\nu-1)$ non-damped
    Block Jacobi iterations followed by an ABGS step.}.

This block iteration, whose $kn$-graph is represented in Figure~\ref{fig:kn-parareal} (right),
does not only link two successive block variables with time index $n+1$ and $n$, but also uses a block with time index $n-1$.
It is not a primary block iteration in the sense of Definition~\ref{def:primaryBlockIteration} anymore.
Although it can be analyzed using generating functions~\cite[Th.~6]{gander2018multigrid}, we focus on primary block iterations here and leave more complex block iterations like this one for future work.
\end{remark}

\subsection{Convergence analysis with GFM bounds}\label{sec:analysisParareal}

In their convergence analysis of \Parareal for non-linear problems~\cite{gander2008nonlinear}, the authors
obtain a double recurrence of the form $e_{n+1}^{k+1} \leq \alpha e_{n}^k + \beta e_{n}^{k+1}$, where $\alpha$ and $\beta$ come from Lipschitz constants and local truncation error bounds.
Using the same notation as in Section~\ref{sec:descrParareal}, with $\alpha=\norm{\F-\G}$ and $\beta =
\norm{\G}$, we find~\cite[Th.~1]{gander2008nonlinear} that
\begin{equation}\label{eq:pararealBoundTrick}
	e_{n+1}^k \leq \delta \frac{\alpha^k}{k!}
	\bar{\beta}^{n-k}\prod_{l=1}^{k}(n+1-l), \quad
	\bar{\beta} = \max(1, \beta).
\end{equation}
This is different from the GFM bound
\begin{equation}\label{eq:pararealBoundBrutal}
  e_{n+1}^{k} \leq
	\delta \frac{\alpha^k}{(k-1)!}
	\sum_{i=0}^{n-k}\prod_{l=1}^{k-1}(i+l)\beta^{i}
\end{equation}
we get when applying Theorem~\ref{th:errBoundPBI} with $\gamma=0$ to
the block iteration of \Parareal.  The difference stems from an
approximation in the proof of~\cite[Th.~1]{gander2008nonlinear} which
leads to the simpler and more explicit bound
in~\eqref{eq:pararealBoundTrick}.  The two bounds are equal when
$\beta=1$, but for $\beta \neq 1$, the GFM bound
in \eqref{eq:pararealBoundBrutal} is sharper.  To illustrate this, we
use the interface formulation of Section~\ref{ex:RungeKutta}: we
set $M:=1$, $\tau_{1}:=1$ and use the block operators
\begin{equation}
	\phiOp := R(\lambda\dt/\ell)^{-\ell}, \quad \chiOp := 1, \quad
	\phiApprox := R_\Delta(\lambda\dt/\ell_\Delta)^{-\ell_\Delta}.
\end{equation}
We solve~\eqref{eq:dahlquist} for $\lambda\in\{i, -1\}$ with
$t\in[0,2\pi]$ and $u_0=1$,
using $N:=10$ blocks, $\ell:=10$ fine time steps per block, the standard 4$^{th}$-order Runge-Kutta method for $\phiOp$
and $\ell_\Delta=2$ coarse time steps per block with Backward Euler
for $\phiApprox$.
Figure~\ref{fig:Parareal} shows the resulting
error (dashed line) at the last time point, the original error
bound~\eqref{eq:pararealBoundTrick}, and the new
bound~\eqref{eq:pararealBoundBrutal}.
\begin{figure}
  \centering
  \includegraphics[width=0.49\linewidth]{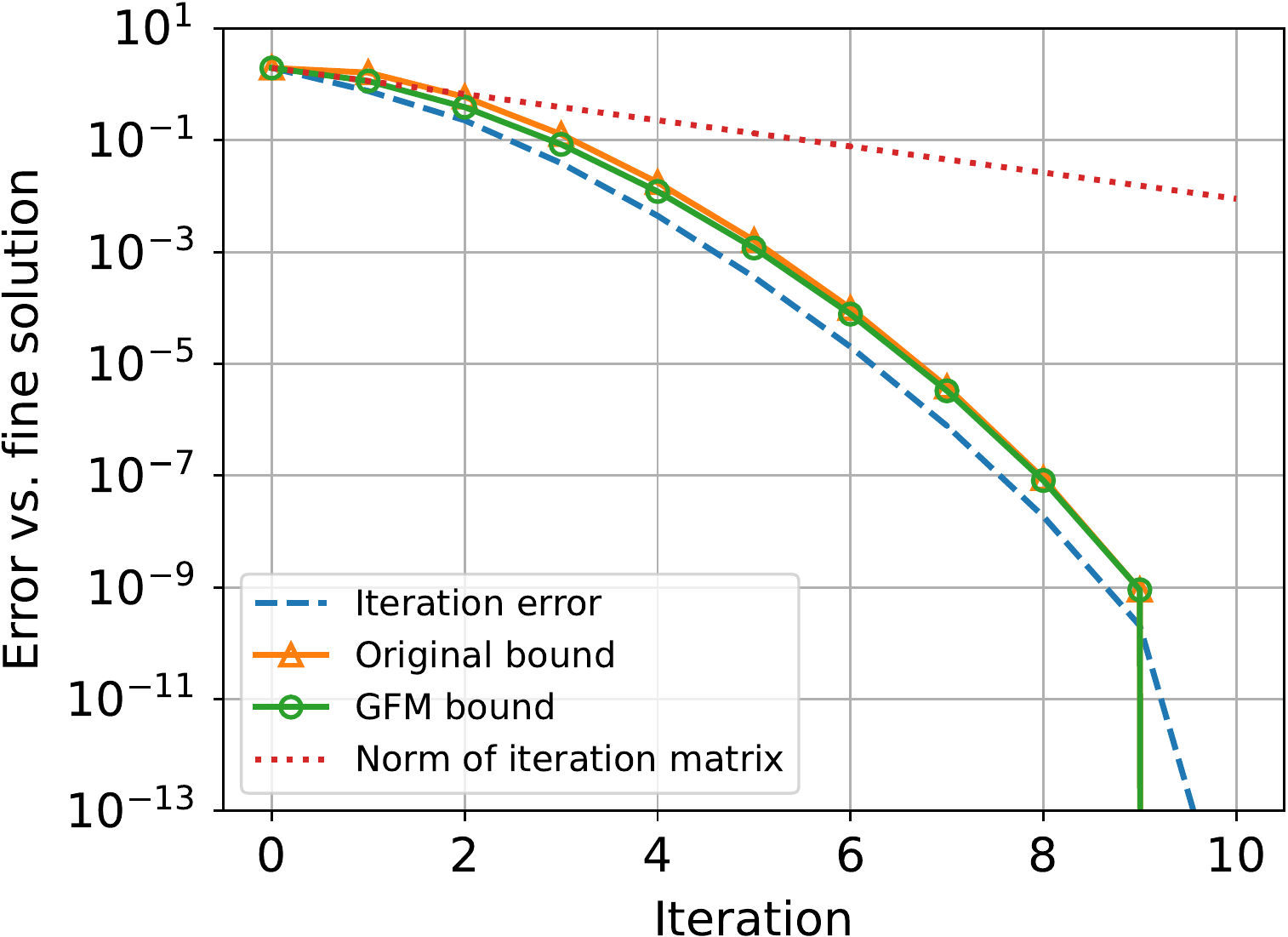}\hfill
  \includegraphics[width=0.49\linewidth]{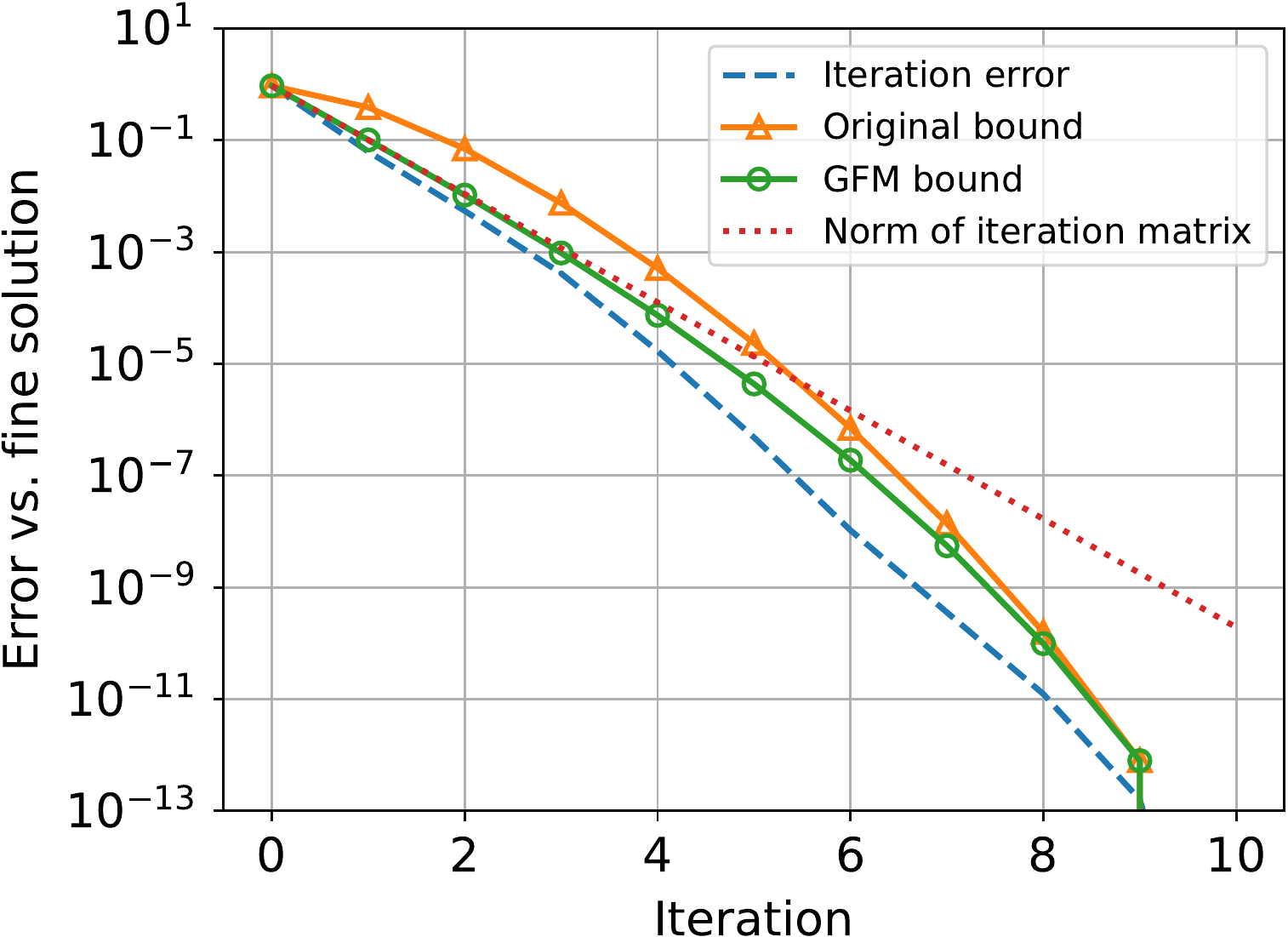}%
  \caption{Error bounds for \Parareal for~\eqref{eq:dahlquist}. Left: $\lambda=i$,
    right: $\lambda=-1$.
    Note that for $\lambda=i$, the GFM-bound and the original one are almost identical.}
  \label{fig:Parareal}
\end{figure}
We also plot the linear bound obtained from the $L^{\infty}$ norm of
the iteration matrix $\RMat_{\Parareal}$ defined just after
\eqref{eq:pararealGlobal}.  For both values of $\lambda$, the
GFM-bounds coincide with the linear bound from $\RMat_{\Parareal}$ for
the first iteration, and the GFM-bound captures the super-linear
contraction in later iterations.  For $\lambda=i$, the old and new
bounds are similar since $\beta$ is close to 1.  However, for
$\lambda=-1$ where $\beta$ is smaller than one, the new bound gives a
sharper estimate of the error, and we can also see that the new bound
captures well the transition from the linear to the super-linear
convergence regime.  On the left in Figure~\ref{fig:Parareal},
\Parareal~seems to converge well for imaginary $\lambda=i$.
This, however, should not be seen as a working example of
\Parareal for a hyperbolic type problem,
but is rather the effect of the relatively good accuracy of the coarse solver using 20 points per wave length for one wavelength present in
    the solution time interval we consider.
Denoting by $\epsilon_\Delta$ the $L_\infty$ error with respect to the exact solution,
the accuracy of the coarse solver ($\epsilon_\Delta=$ 6.22e-01)
allows the \Parareal error to reach the fine solver
error ($\epsilon_\Delta=$ 8.16e-07) in $K=8$ iterations.
Since the ideal parallel speedup of \Parareal,
neglecting the coarse solver cost, is bounded by $N/K=1.25$
\cite[Sec.~4]{aubanel2011scheduling},
this indicates however almost no speedup in practical applications (see also \cite{GoetschelEtAl2021}).
If we increase the coarse solver error,
for instance by multiplying $\lambda$ by a factor $4$
to have now four times more wavelength in the domain, and only
12.5 points per wavelength resolution in the coarse solver, the
convergence of \Parareal deteriorates massively, as we can see in
Figure~\ref{fig:Parareal-2} (left), while this is not the case for
the purely negative real fourfold $\lambda=-4$.

This illustrates how Parareal has great convergence difficulties for hyperbolic problems,
already well-documented in the literature see \eg \cite{gander2008analysis2,gander2020toward}.
This is analogous to the difficulties due to the pollution error
and damping in multi-grid methods when solving medium to high
frequency associated time harmonic problems,
see \cite{elman2001multigrid,ernst2012difficult, ernst2013multigrid,gander2015applying,cocquet2017large}
and references therein.
\begin{figure}
	\centering
	\includegraphics[width=0.49\linewidth]{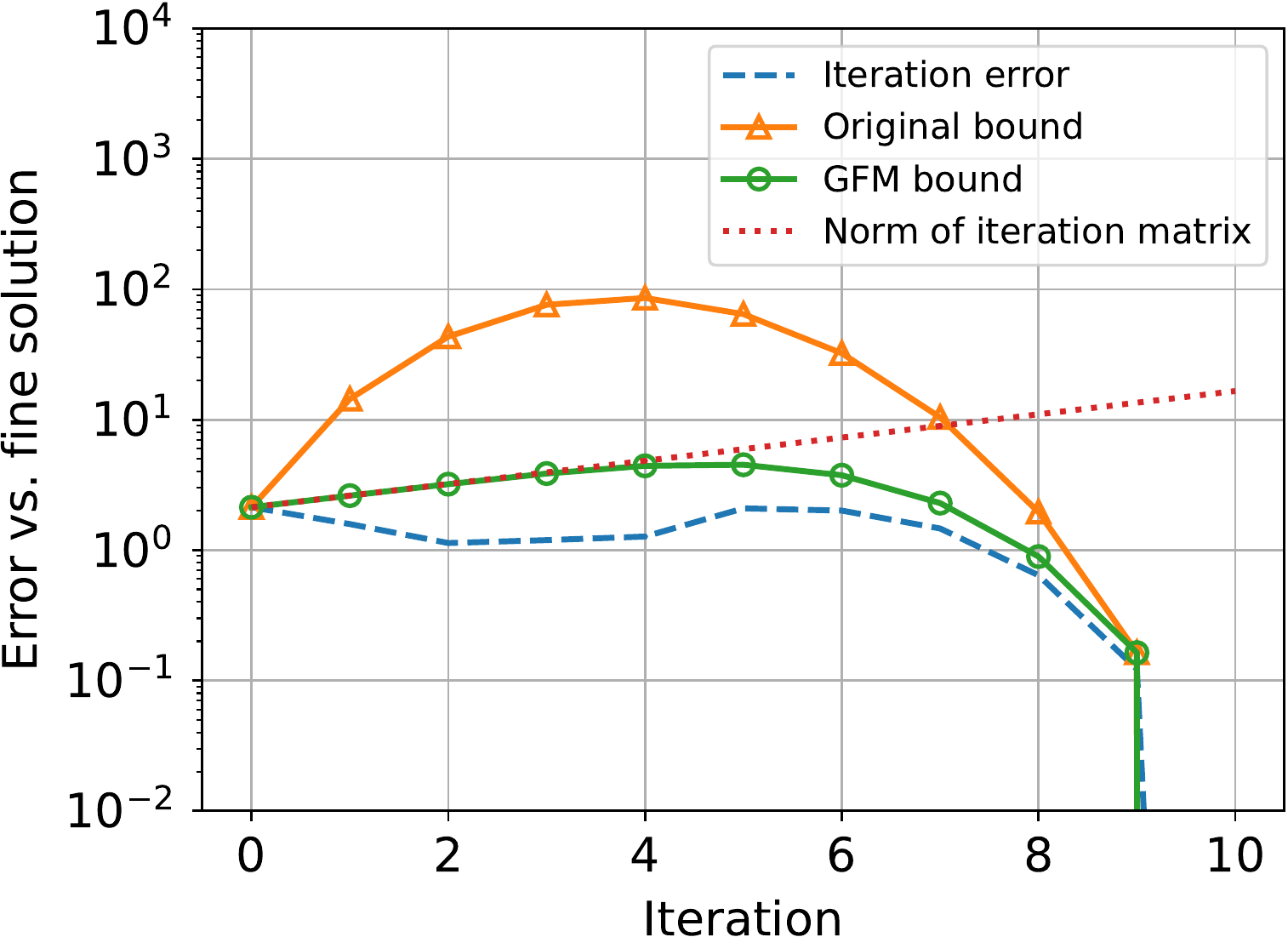}\hfill
	\includegraphics[width=0.49\linewidth]{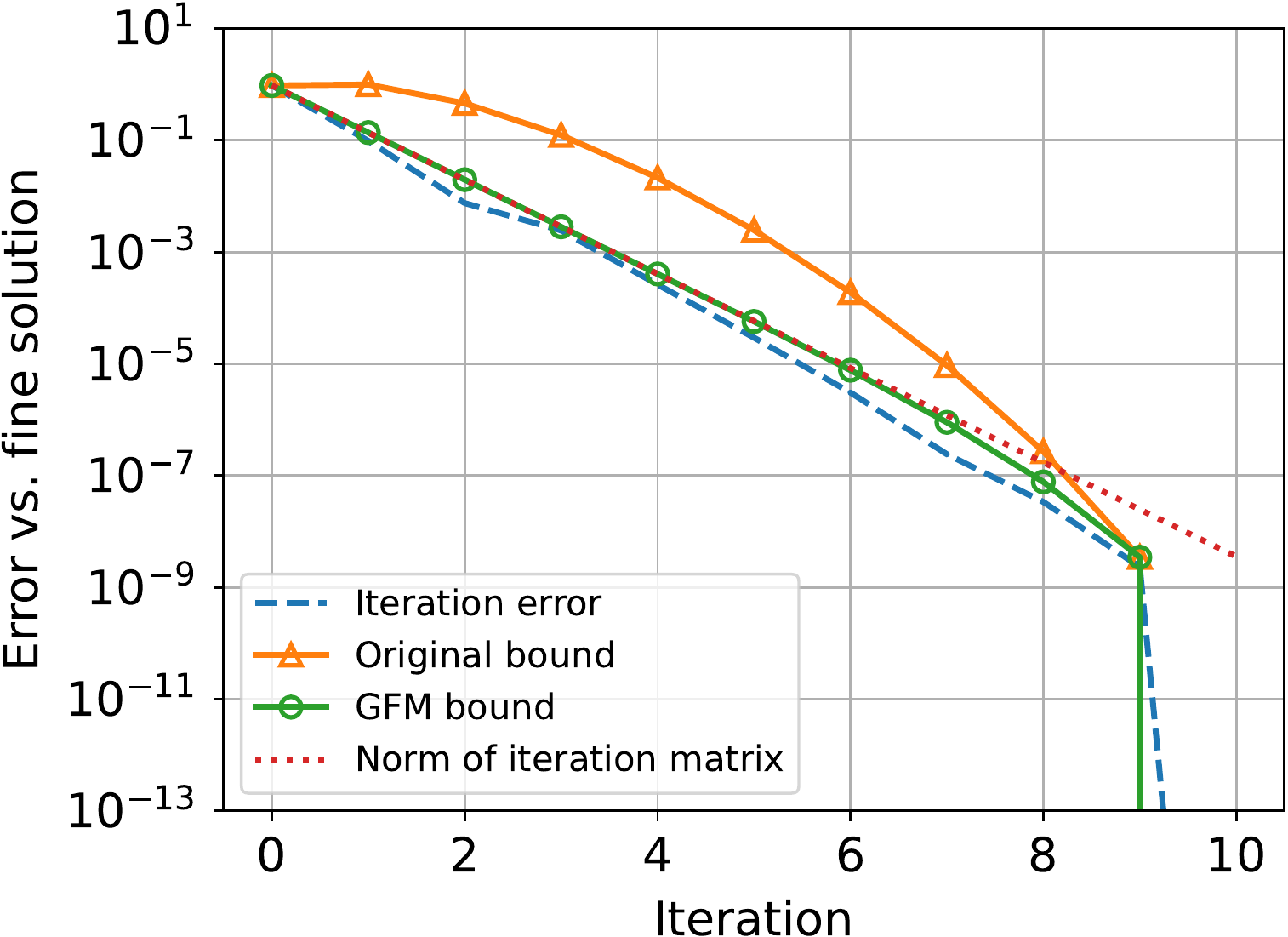}%
	\caption{Error bounds for \Parareal for~\eqref{eq:dahlquist}.  Left: $\lambda=4i$,
		right: $\lambda=-4$.}
	\label{fig:Parareal-2}
\end{figure}

\section{Writing two-level Time Multi-Grid as a block iteration}\label{sec:STMG}

The idea of time multi-grid (\TMG) goes back to the 1980s and
1990s~\cite{burmeister1991time,hackbusch1984parabolic,murata1991parabolic}.
Furthermore, not long after \Parareal was introduced, it was shown to be
equivalent to a time multi-grid method, independently of the type of
approximation used for the coarse solver~\cite{gander2007analysis}.
This inspired the development of other time multi-level
methods, in particular \MGRIT~\cite{falgout2014parallel}.
However, \Parareal and \MGRIT are usually viewed as iterations acting on
values located at the block interface,
while \TMG-based algorithms, in particular \STMG~\cite{gander2016analysis},
use an iteration updating volume values
(\ie all fine time points in the time domain).
In this section, we focus on a generic description of \TMG, and show how to
write its two-level form applied to the Dahlquist problem as block iteration.
In particular, we will show in Section~\ref{sec:PFASST} that \PFASST can be
expressed as a specific variant of \TMG.
The extension of this analysis to more levels and comparison with multi-level
\MGRIT is left for future work.

\subsection{Definition of a coarse block problem for Time Multi-Grid}\label{sec:coarseProblem}

To build a coarse problem, we consider a coarsened version of the global
problem~\eqref{eq:globalProblem},
with a $\ACoarse$ matrix having $N\cdot \MCoarse$ rows instead
of $N\cdot M$ for $\AMat$.
For each of the $N$ blocks, let $(\tau^C_{m})_{1\leq m\leq\MCoarse}$ be
the normalized $\MCoarse$ grid points\footnote{
    Those do not need to be a subset
    of the fine block grid points, although they usually are in
    applications.}
of a \emph{coarser} block discretization, with $\MCoarse < M$.

We can define a coarse block operator
$\phiCoarse$ by using the same time integration method as for
$\phiOp$ on every block, but with fewer time points.
This is equivalent to geometric coarsening used for $h$-multigrid (or
geometric multigrid \cite{trottenberg2001multigrid}), \eg when using one
time-step of a Runge-Kutta method between each time grid point.
It can also be equivalent to spectral coarsening used for $p$-multigrid
(or spectral element multigrid \cite{ronquist1987spectral}),
\eg when one step of a collocation method on $M$ points is used within each
block (as for \PFASST, see Section \ref{sec:PFASST-descr}).

We also consider the associated transmission operator $\chiCoarse$,
and denote by $\uCoarse_n$ the block variable on this coarse time
block, which satisfies
\begin{equation}
    \phiCoarse(\uCoarse_1) = \chiCoarse\TFtoC(u_0\vect{\ones}),\quad
    \phiCoarse\uCoarse_{n+1} = \chiCoarse\uCoarse_{n}
        \quad n = 1, 2, \dots, N-1.
\end{equation}
Let $\uCoarse$ be the global coarse variable that solves
\begin{equation}\label{eq:globalCoarseProblem}
    \ACoarse\uCoarse :=
    \begin{pmatrix}
        \phiCoarse & & &\\
        -\chiCoarse & \phiCoarse & &\\
        & \ddots & \ddots &\\
        & & -\chiCoarse & \phiCoarse
    \end{pmatrix}
    \begin{bmatrix}
        \uCoarse_{1}\\\uCoarse_{2}\\\vdots\\\uCoarse_{N}
    \end{bmatrix}
    =
    \begin{bmatrix}
        \chiCoarse\TFtoC(u_0\vect{\ones})\\0\\\vdots\\0
    \end{bmatrix}
    =: \vect{f^\CoarseId}.
\end{equation}
$\TFtoC$ is a block restriction operator, i.e. a transfer matrix
from a fine (F) to a coarse (C) block discretization.
Similarly, we have a block prolongation operator $\TCtoF$,
i.e. a transfer matrix from a coarse (C) to a fine (F) block discretization.

\begin{remark}
  While both $\phiCoarse$ and $\phiApprox$ are approximations of
  the fine operator $\phiOp$, the main difference between $\phiCoarse$
  and $\phiApprox$ is the size of the vectors they can be applied to
  ($\MCoarse$ and $M$).  Furthermore, $\phiCoarse$ itself does need
  the transfer operators $\TCtoF$ and $\TFtoC$ to compute approximate
  values on the fine time points, while $\phiApprox$ alone is
  sufficient (even if it can hide some restriction and interpolation
  process within).  However, the definition of a coarse grid
  correction in the classical multi-grid formalism needs this separation
  between transfer and coarse operators (see
  \cite[Sec.~2.2.2]{trottenberg2001multigrid}), which limits the use of
  $\phiApprox$ and requires the introduction of $\phiCoarse$.
\end{remark}

\subsection{Block iteration of a Coarse Grid Correction}\label{sec:coarseGridCorrection}

Let us consider a standalone Coarse Grid Correction (CGC),
without pre- or post-smoothing\footnote{The CGC is not convergent by itself
without a smoother.},
of a two-level multi-grid iteration~\cite{hackbusch2013multi} applied to
\eqref{eq:globalProblem}.
One CGC step applied to \eqref{eq:globalProblem} can be written as
\begin{equation}\label{eq:globalCoarseCorrection}
    \uvect^{k+1} = \uvect^{k} + \TCtoFBar \ACoarse^{-1} \TFtoCBar
    (\vect{f}-\AMat\uvect^k),
\end{equation}
where $\TCtoFBar$ denotes the block diagonal matrix formed with
$\TCtoF$ on the diagonal, and similarly for $\TFtoCBar$.
When splitting~\eqref{eq:globalCoarseCorrection} into two steps,
\begin{gather}
    \ACoarse\vect{d} = \TFtoCBar (\vect{f}-\AMat\uvect^k),\\
    \uvect^{k+1} = \uvect^{k} + \TCtoFBar\vect{d},
\end{gather}
the CGC term (or defect) $\vect{d}$ appears explicitly.
Expanding the two steps for $n>0$ into a block formulation and inverting $\phiCoarse$ leads to
\begin{align}
    \vect{d}_{n+1} &= \phiCoarse^{-1}\TFtoC\chiOp\uvect^{k}_{n}
    - \phiCoarse^{-1}\TFtoC\phiOp\uvect^{k}_{n+1}
    + \phiCoarse^{-1}\chiCoarse\vect{d}_{n}, \label{eq:coarseCorrection}\\
    \uvect_{n+1}^{k+1} &= \uvect_{n+1}^{k} + \TCtoF\vect{d}_{n+1}.\label{eq:addCoarseCorrection}
\end{align}
Now we need the following simplifying assumption.
\begin{assumption}\label{ass:transfersOp}
    Prolongation $\TCtoF$ followed by restriction $\TFtoC$ leaves
    the coarse block variables unchanged, \ie
    \begin{equation}
        \TFtoC \TCtoF = \eyeMat.
    \end{equation}
\end{assumption}
This condition is satisfied in many situations
(\eg restriction with standard injection on a coarse subset of the fine
    points, or polynomial interpolation with any possible coarse block discretization)%
\footnote{In some situations, \eg when the transpose of
    linear interpolation is used for
    restriction (full-weighting), we do not get the identity in Assumption~\ref{ass:transfersOp} but an invertible matrix.
    The same simplifications can be done, except one must take into	account the inverse of $(\TFtoC\TCtoF)$.}.
Using it in~\eqref{eq:addCoarseCorrection} for block index $n$ yields
\begin{equation}
    \vect{d}_{n} = \TFtoC\left(\uvect_{n}^{k+1} - \uvect_{n}^{k}\right).
\end{equation}
Inserting $\vect{d}_n$ into~\eqref{eq:coarseCorrection} on the right and the resulting $\vect{d}_{n+1}$ into~\eqref{eq:addCoarseCorrection} leads to
\begin{equation}\label{eq:twoLevel}
    \uvect_{n+1}^{k+1} = (\Imat-\TCtoF\phiCoarse^{-1}\TFtoC\phiOp)\uvect^{k}_{n+1}
    + \TCtoF\phiCoarse^{-1}\chiCoarse\TFtoC\uvect_{n}^{k+1}
    + \TCtoF\phiCoarse^{-1}\Delta_\chi\uvect^{k}_{n},
\end{equation}
with $\Delta_\chi := \TFtoC\chiOp - \chiCoarse\TFtoC$.
\begin{figure}
    \centering\hfil
    \includegraphics[height=0.2\linewidth]{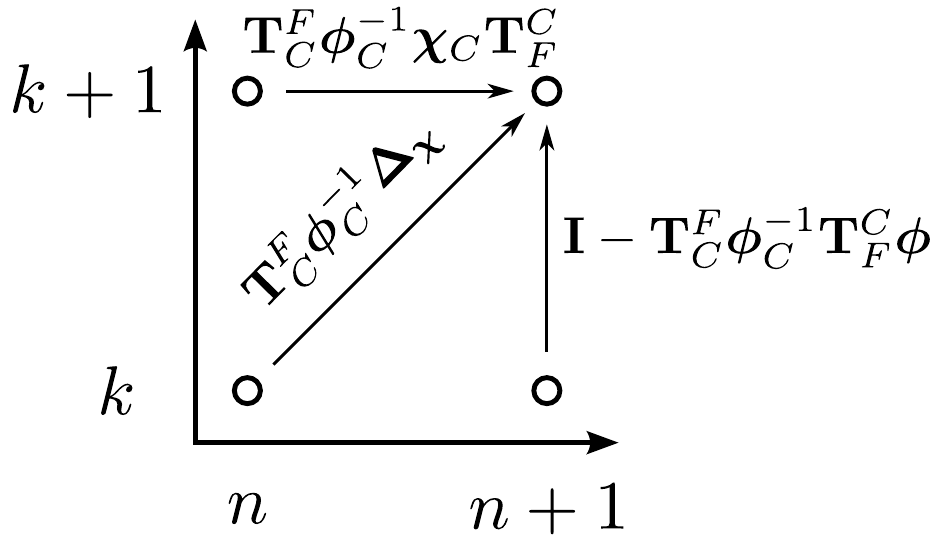}\hfil%
    \includegraphics[height=0.2\linewidth]{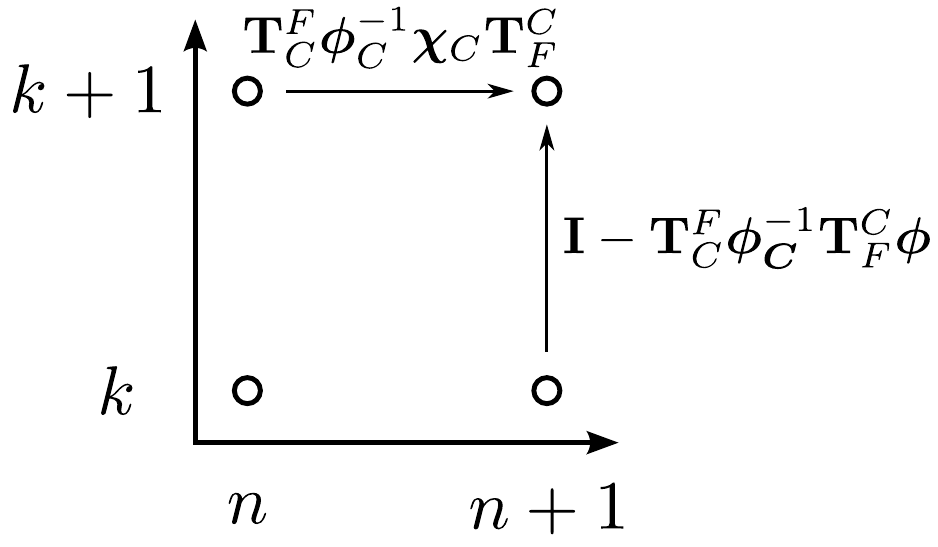}\hfil
    \caption{$kn$-graphs for the CGC block iteration, with Assumption~\ref{ass:transfersOp} only (left),
    and with both Assumptions~\ref{ass:transfersOp} and \ref{ass:deltaChi} (right).}
    \label{fig:kn-CGC}
\end{figure}
This is a primary block iteration in the sense of
Definition~\ref{def:primaryBlockIteration},
and we give its $kn$-graph
in Figure~\ref{fig:kn-CGC} (left).
We can simplify it further using a second assumption:
\begin{assumption}\label{ass:deltaChi}
    We consider operators $\TFtoC$, $\chiOp$ and $\chiCoarse$
    such that
    \begin{equation}
        \Delta_\chi = \TFtoC\chiOp - \chiCoarse\TFtoC = 0.
    \end{equation}
    This holds for classical time-stepping methods when both left and
    right time sub-interval boundaries are included in the block
    variables, or for collocation methods using Radau-II or Lobatto type
    nodes.
\end{assumption}
This last assumption is important to define \PFASST
(\cf Section~\ref{sec:PFASST-descr} and see
Bolten et al.~\cite[Remark 1]{bolten2017multigrid}
for more details) and simplifies the analysis
of \TMG, as both methods use this block iteration.
Then, \eqref{eq:twoLevel} reduces to
\begin{equation}\label{eq:twoLevelSimple}
    \uvect_{n+1}^{k+1} = (\Imat-\TCtoF\phiCoarse^{-1}\TFtoC\phiOp)\uvect^{k}_{n+1}
    + \TCtoF\phiCoarse^{-1}\TFtoC\chiOp\uvect_{n}^{k+1}.
\end{equation}
Again, this is a primary block iteration for which the $kn$-graph is given in Figure~\ref{fig:kn-CGC} (right).
It satisfies the consistency condition\footnote{Note that the consistency condition is satisfied even without assumption~\ref{ass:deltaChi}.}~\eqref{eq:primaryBlockIterCondition} since
$
((\Imat-\TCtoF\phiCoarse^{-1}\TFtoC\phiOp) - \Imat)
\phiOp^{-1}\chiOp
+ \TCtoF\phiCoarse^{-1}\TFtoC\chiOp = 0.
$

\subsection{Two-level Time Multi-Grid}\label{sec:two-levelTMG}

Gander and Neum\"uller introduced \STMG for discontinuous Galerkin approximations in time~\cite{gander2016analysis}, which leads to a similar system as~\eqref{eq:globalProblem}.
We describe the two-level approach for general time discretizations, following their multi-level description~\cite[Sec. 3]{gander2016analysis}.
Consider a coarse problem defined as in Section~\ref{sec:coarseGridCorrection} and a damped
block Jacobi smoother as in Section~\ref{ex:blockJacobi} with relaxation parameter $\omega$.
Then, a two-level \TMG iteration requires the following steps:
\begin{enumerate}
    \item $\nu_1$ pre-relaxation steps \eqref{eq:jacobiRelaxationGlobal} with block Jacobi smoother,
    \item one CGC \eqref{eq:globalCoarseCorrection} inverting the coarse grid operators,
    \item $\nu_2$ post-relaxation steps \eqref{eq:jacobiRelaxationGlobal} with the block Jacobi smoother,
\end{enumerate}
each corresponding to a block iteration.
If we combine all these block iterations we do not obtain a primary block
iteration but a more complex expression, of which the analysis is beyond the
scope of this paper.
However, a primary block iteration in the sense of Definition~\ref{def:primaryBlockIteration} is obtained when
\begin{itemize}
    \item Assumption~\ref{ass:deltaChi} holds, so that
    $\Delta_\chi=0$,
    \item only one pre-relaxation step is used, $\nu_1=1$,
    \item and no post-relaxation step is considered, $\nu_2=0$.
\end{itemize}
Then, the two-level iteration reduces to the two block updates
from~\eqref{eq:jacobiRelaxation} and~\eqref{eq:twoLevelSimple},
\begin{gather}
    \uvect_{n+1}^{k+1/2} = (1-\omega)\uvect_{n+1}^k +
    \omega\phiOp^{-1}\chiOp \uvect_{n}^k,
    \label{eq:STMGSmoother}\\
    \uvect_{n+1}^{k+1} =
    \left(\Imat-\TCtoF\phiCoarse^{-1}\TFtoC\phiOp\right)\uvect^{k+1/2}_{n+1}
    + \TCtoF\phiCoarse^{-1}\chiCoarse\TFtoC\uvect_{n}^{k+1},
    \label{eq:STMGTwoGrid}
\end{gather}
using $k+1/2$ as intermediate index.
Combining~\eqref{eq:STMGSmoother} and~\eqref{eq:STMGTwoGrid} leads to the primary block iteration
\begin{equation}\label{eq:STMG-generic}
    \uvect_{n+1}^{k+1} =
    \left(\Imat-\TCtoF\phiCoarse^{-1}\TFtoC\phiOp\right)
    \left[(1-\omega)\uvect_{n+1}^k +
    \omega\phiOp^{-1}\chiOp \uvect_{n}^k\right]
    + \TCtoF\phiCoarse^{-1}\chiCoarse\TFtoC\uvect_{n}^{k+1}.
\end{equation}
If $\omega \neq 1$, all block operators in this primary block iteration
are non-zero, and applying Theorem~\ref{th:errBoundPBI}
leads to the error bound~\eqref{Bound4}.
Since the latter is similar to the one obtained for~\PFASST in
Section~\ref{sec:PFASST-analysis}, we leave its comparison with numerical
experiments to Section \ref{sec:PFASST}.
For $\omega=1$ we get the simplified iteration
\begin{equation}\label{eq:STMG-simple}
    \uvect_{n+1}^{k+1} =
    \left(\phiOp^{-1}\chiOp-\TCtoF\phiCoarse^{-1}\TFtoC\chiOp\right)
    \uvect_{n}^k
    + \TCtoF\phiCoarse^{-1}\chiCoarse\TFtoC\uvect_{n}^{k+1},
\end{equation}
and the following result:
\begin{proposition}\label{prop:equiv-TMG-Parareal}
    Consider a CGC as in
    Section~\ref{sec:coarseGridCorrection}, such that the prolongation and
    restriction operators (in time) satisfy Assumption~\ref{ass:transfersOp}.
    If Assumption~\ref{ass:deltaChi} also holds and only one block Jacobi
    pre-relaxation step \eqref{eq:jacobiRelaxationGlobal} with $\omega=1$ is used before
    the CGC,
    then two-level \TMG is equivalent to \Parareal,
    where the coarse solver $\mathcal{G}$ uses the same time integrator as the
    fine solver $\mathcal{F}$ but with larger time steps, \ie
    \begin{equation}
        \mathcal{G} := \TCtoF\phiCoarse^{-1}\TFtoC\chiOp.
    \end{equation}
\end{proposition}
This is a particular case of a result obtained before by
Gander~\cite[Theorem~3.1]{gander2007analysis} but is presented here in the
context of our GFM framework and the definition of \Parareal given in
Section~\ref{sec:descrParareal}.
In particular, it shows that the simplified two-grid iteration
on~\eqref{eq:globalProblem} is equivalent to the preconditioned fixed-point
iteration~\eqref{eq:pararealGlobal} of \Parareal if some conditions are met
and  $\phiApprox^{-1} := \TCtoF\phiCoarse^{-1}\TFtoC$ is used as
the approximate integration operator\footnote{
    Note that, even if $\TCtoF\phiCoarse^{-1}\TFtoC$ is not invertible,
    this abuse of notation is possible as \eqref{eq:pararealGlobal} requires
    an approximation of $\phiOp^{-1}$ rather than an approximation of
    $\phiOp$ itself.}.
However, the \TMG iteration here updates also the fine time point values,
using $\TCtoF$ to interpolate the coarse values computed with
$\phiCoarse$, hence applying the \Parareal update
\emph{to all volume values}.
This is the only ``difference" with the original definition of \Parareal in
\cite{lions2001parareal}, where the update is only applied to the interface
value between blocks.

One key idea of \STMG that we have not described yet is the block
diagonal Jacobi smoother used for relaxation.  Even if its diagonal
blocks use a time integration operator identical to those of the
fine problem (hence requiring the inversion of $\phiOp$), their
spatial part in \STMG is approximated using one V-cycle multi-grid
iteration in space based on a pointwise smoother
\cite[Sec.~4.3]{gander2016analysis}. We do not cover this aspect in
our description of \TMG here, since we focus on time only,
but describe in the next section a similar approach that is used
for \PFASST.

\section{Writing \PFASST as a block iteration}
\label{sec:PFASST}

\PFASST is also based on a \TMG approach using an approximate relaxation step,
but the approximation of the block Jacobi smoother is
\emph{done in time and not in space, in contrast to \STMG}.  In
addition, the CGC in \PFASST is also approximated,
\ie there is no direct solve on the coarse level to compute the
CGC as in \STMG. One \PFASST iteration is therefore a
combination of an \emph{Approximate Block Jacobi} (ABJ) smoother, see
Section~\ref{sec:BJ-SDC}, followed by one (or more) ABGS iteration(s)
of Section~\ref{ex:ABGS} on the coarse level to approximate the
CGC~\cite[Sec.~3.2]{emmett2012toward}.
While we describe only the two-level variant, the algorithm can use more levels~\cite{emmett2012toward,SpeckEtAl2012}.
The main component of \PFASST is the approximation of the time integrator blocks using Spectral Deferred Corrections
(SDC)~\cite{dutt2000spectral}, from which its other key
components (ABJ and ABGS) are built.
Hence we first describe how SDC is used to define an ABGS iteration in Section~\ref{sec:SDC},
then ABJ in Section~\ref{sec:BJ-SDC},
and finally \PFASST in Section~\ref{sec:PFASST-descr}.

\subsection{Approximate Block Gauss-Seidel with SDC}\label{sec:SDC}
SDC can be seen as a preconditioner when integrating the ODE problem
\eqref{eq:dahlquist} with collocation methods, see Section~\ref{ex:collocation}.
Consider the block operators
\begin{equation}\label{eq:blockCollocation}
\phiOp := (\Imat-\Qmat), \quad
\chiOp := \Hmat \quad \Longrightarrow \quad
(\Imat-\Qmat) \vect{u}_{n+1} = \Hmat \vect{u}_{n}.
\end{equation}
SDC approximates the quadrature matrix $\Qmat$ by
\begin{equation}
\QDelta = \lambda\dt \left(\tilde{q}_{m,j}\right), \quad \tilde{q}_{m,j} = \int_{0}^{\tau_m} \tilde{l}_{j}(s)ds,
\end{equation}
where $\tilde{l}_{j}$ is an approximation of the Lagrange polynomial $l_j$.
Usually, $\QDelta$ is lower triangular~\cite[Sec~3]{ruprecht2016spectral} and easy to invert\footnote{
	The notation $\QDelta$ was chosen instead of $\matr{\tilde{Q}}$ for consistency with
	the literature, \cf \cite{ruprecht2016spectral,bolten2017multigrid,bolten2018asymptotic}.}.
This approximation is used to build the preconditioned iteration
\begin{equation}\label{eq:blockSDC}
\vect{u}_{n+1}^{k+1} = \vect{u}_{n+1}^{k} +
[\Imat - \QDelta]^{-1}\left(
\Hmat\vect{u}_{n}
- (\Imat-\Qmat)\vect{u}_{n+1}^k
\right)
\end{equation}
to solve \eqref{eq:blockCollocation}, with $\vect{u}_{n+1}$ as unknown.
We obtain the generic preconditioned iteration for one block,
\begin{equation}\label{eq:blockSDCGeneric}
    \vect{u}_{n+1}^{k+1} = \left[\eyeMat - \phiApprox^{-1}\phiOp\right]\vect{u}_{n+1}^k
    + \phiApprox^{-1} \chiOp\vect{u}_{n} \quad \text{with} \quad \phiApprox:=\Imat - \QDelta.
\end{equation}
This shows that SDC inverts the $\phiOp$ operator approximately
  using $\phiApprox$ block by block to solve the global
  problem~\eqref{eq:globalProblem}, \ie it fixes an $n$
  in~\eqref{eq:blockSDCGeneric}, iterates over $k$ until convergence,
  and then increments $n$ by one. Hence SDC gives a natural way
  to define an approximate block integrator $\phiApprox$ to be used
  to build ABJ and ABGS iterations.
Defining the ABGS iteration \eqref{eq:ABGSGlobal} of Section~\ref{ex:ABGS}
using the SDC block operators gives the block updating formula
\begin{equation}\label{eq:blockIterationBGS}
    \vect{u}_{n+1}^{k+1} = \vect{u}_{n+1}^{k} + [\Imat -
    \QDelta]^{-1}\left( \Hmat\vect{u}_{n}^{k+1} -
    (\Imat-\Qmat)\vect{u}_{n+1}^k \right),
\end{equation}
which we call \emph{Block Gauss-Seidel SDC} (BGS-SDC), very similar
to~\eqref{eq:blockSDC} except that we use the
new iterate $\uvect_{n}^{k+1}$ and not the
converged solution $\uvect_{n}$. This is a primary
block iteration in the sense of Definition~\ref{def:primaryBlockIteration}
with
\begin{equation}
    \begin{split}\label{eq:blockCoefficientsBGSSDC}
        \BMat^0_1 &:= \Imat - [\Imat - \QDelta]^{-1}(\Imat-\Qmat) = [\Imat - \QDelta]^{-1}(\Qmat-\QDelta),\\
        \BMat^0_0 &:= 0,\quad \BMat^1_0:= [\Imat - \QDelta]^{-1}\Hmat,
    \end{split}
\end{equation}
and its $kn$-graph is shown in Figure \ref{fig:kn-BJ-BGS} (right).

\subsection{Approximate Block Jacobi with SDC}
\label{sec:BJ-SDC}
Here we solve the global problem~\eqref{eq:globalProblem} using a preconditioner that can be easily parallelized (Block Jacobi) and
combine it with the approximation of the collocation operator $\phiOp$ by $\phiApprox$ defined in~\eqref{eq:blockCollocation} and~\eqref{eq:blockSDCGeneric}.
This leads to the \emph{global} preconditioned iteration
\begin{equation}\label{eq:approxJacobiGlobal}
\uvect^{k+1} = \uvect^k + \MJac^{-1}(\vect{f}-\AMat\uvect^k),\quad
\MJac =
\begin{bmatrix}
\phiApprox & & \\
& \phiApprox & \\
& & \ddots
\end{bmatrix}.
\end{equation}
This is equivalent to the block Jacobi relaxation in
Section~\ref{ex:blockJacobi} with $\omega=1$, except that the block
operator $\phiOp$ is approximated by $\phiApprox$.  Using the
SDC block operators~\eqref{eq:blockCollocation} gives the block
updating formula
\begin{equation}\label{eq:blockIterationBJ}
\vect{u}_{n+1}^{k+1} = \vect{u}_{n+1}^{k} +
[\Imat - \QDelta]^{-1}\left(
\Hmat\vect{u}_{n}^{k}
- (\Imat-\Qmat)\vect{u}_{n+1}^k
\right),
\end{equation}
which we call \emph{Block Jacobi SDC} (BJ-SDC).
This is a primary block iteration with
\begin{equation}
\begin{split}
\BMat^0_1 &:= \Imat - [\Imat - \QDelta]^{-1}(\Imat-\Qmat) = [\Imat - \QDelta]^{-1}(\Qmat-\QDelta),\\
\BMat^0_0 &:= [\Imat - \QDelta]^{-1}\Hmat,\quad \BMat^1_0 := 0.
\end{split}
\end{equation}
Its $kn$-graph is shown in Figure \ref{fig:kn-BJ-BGS} (left).
\begin{figure}
	\centering
	\includegraphics[height=0.2\linewidth]{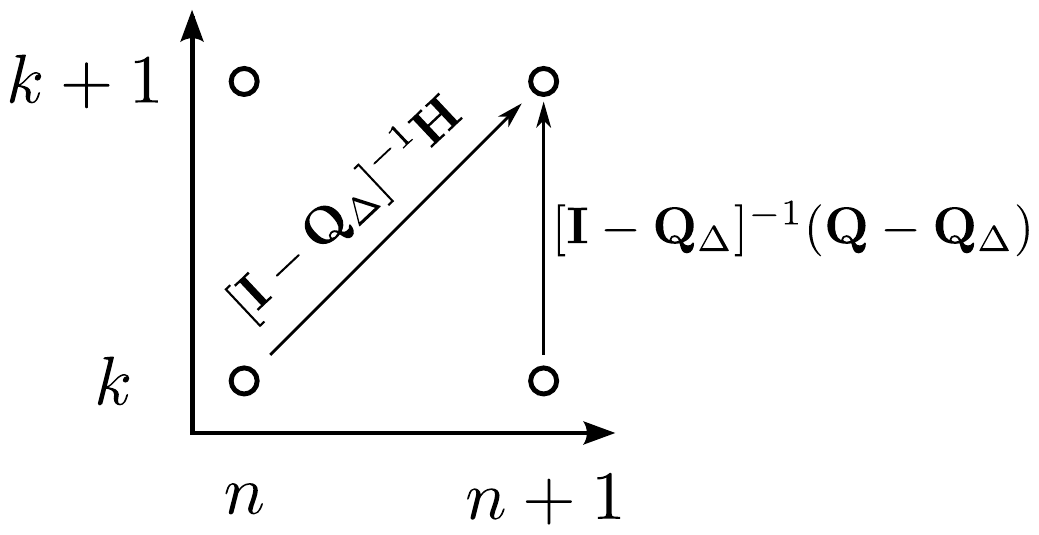}\hfil%
	\includegraphics[height=0.2\linewidth]{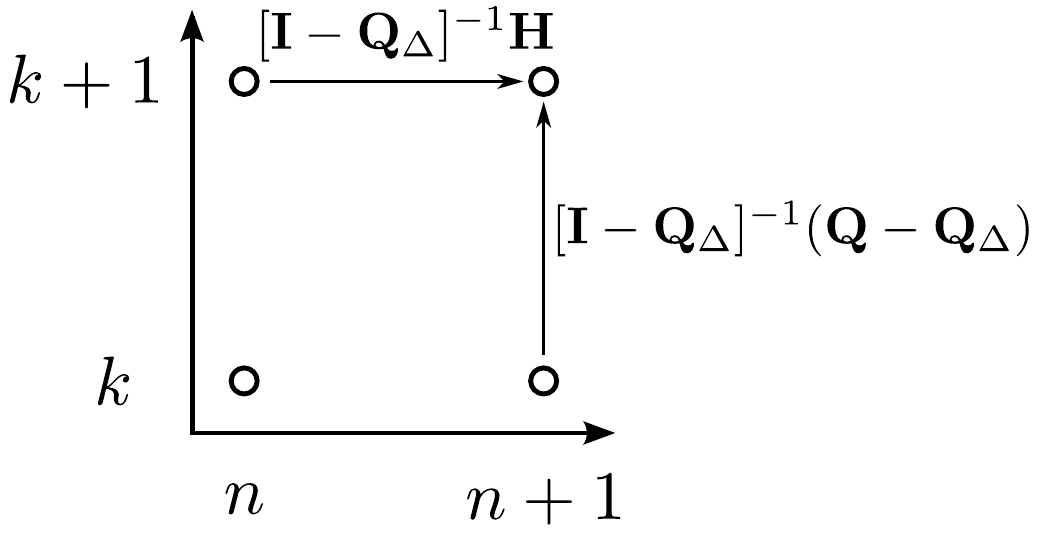}%
	\caption{$kn$-graphs for Block Jacobi SDC (left)
		and Block Gauss-Seidel SDC (right).}
	\label{fig:kn-BJ-BGS}
\end{figure}
This block iteration can be written in the more generic form
\begin{equation}\label{eq:blockJacobiGeneric}
\vect{u}_{n+1}^{k+1} = \left[\eyeMat - \phiApprox^{-1} \phiOp\right] \vect{u}_{n+1}^{k} +
\phiApprox^{-1}\chiOp\vect{u}_{n}^{k}.
\end{equation}
This is similar to~\eqref{eq:blockSDCGeneric} except that we use the
current iterate $\uvect_{n}^k$ from the previous block and not the
converged solution $\uvect_{n}$.
Note that $\phiOp$ and $\phiApprox$ do not need to correspond to
the SDC operators~\eqref{eq:blockCollocation} and~\eqref{eq:blockSDCGeneric}.
This block iteration does not explicitly depend on the use of SDC,
hence the name \emph{Approximate Block Jacobi} (ABJ).

\subsection{\PFASST}\label{sec:PFASST-descr}

We now give a simplified description of PFASST~\cite{emmett2012toward}
applied to the Dahlquist problem \eqref{eq:dahlquist}.  In particular,
this corresponds to doing only one SDC sweep on the coarse level.  To
write \PFASST as a block iteration, we first build the coarse level
as in Section~\ref{sec:coarseGridCorrection}.  From that we can form
the $\QTilde$ quadrature matrix
associated with the coarse nodes and the coarse matrix $\HTilde$,
as we would have done if we were using the collocation method of
    Section~\ref{ex:collocation} on the coarse nodes.
This leads  to the definition of the $\phiCoarse$ and $\chiCoarse$ operators for the coarse level, combined with the transfer operators $\TFtoC$ and $\TCtoF$,
from which we can build the global matrices $\ACoarse$, $\TCtoFBar$ and $\TFtoCBar$, see Section~\ref{sec:coarseGridCorrection}.
Then we build the two-level \PFASST iteration by defining a specific smoother and a modified CGC.

The smoother corresponds to a Block Jacobi SDC iteration
\eqref{eq:blockIterationBJ}
from Section~\ref{sec:BJ-SDC} to produce an intermediate solution
\begin{equation}\label{PFASSTSmoothingStep}
    \uvect_{n+1}^{k+1/2} = [\Imat-\QDelta]^{-1}(\Qmat-\QDelta)\uvect_{n+1}^k
    + [\Imat-\QDelta]^{-1}\Hmat\uvect_{n}^k,
\end{equation}
 denoted with iteration index $k+1/2$.
Using a CGC as in Section~\ref{sec:coarseGridCorrection} would provide the global update formula
\begin{gather}
    \ACoarse\vect{d} = \TFtoCBar (\vect{f}-\AMat\uvect^{k+1/2}),\label{eq:PFASST-directCGC}\\
    \uvect^{k+1} = \uvect^{k+1/2} + \TCtoFBar\vect{d}.\label{eq:PFASST-prolong}
\end{gather}
Instead of a direct solve with $\ACoarse$ to compute the defect $\vect{d}$, in PFASST one uses $L$ Block Gauss-Seidel SDC iterations (or sweeps) to approximate it.
Then \eqref{eq:PFASST-directCGC} becomes
\begin{equation}
    \MGSTilde \vect{d}^{\ell} =
    (\MGSTilde - \ACoarse)\vect{d}^{\ell-1}
    + \TFtoCBar (\vect{f}-\AMat\uvect^{k+1/2}),
    \quad \vect{d}^0 = 0,\quad \ell \in \{1,..,L\},
\end{equation}
and reduces for one sweep only ($L=1$) to
\begin{equation}
    \MGSTilde \vect{d} = \TFtoCBar (\vect{f}-\AMat\uvect^{k+1/2}),
    \quad\MGSTilde =
    \begin{bmatrix}
        \phiApproxCoarse & & \\
        - \chiCoarse & \phiApproxCoarse & \\
        & \ddots & \ddots
    \end{bmatrix}.
\end{equation}
Here $\MGSTilde$ correspond to the $\MGS$ preconditioning matrix, but written on the coarse level using an SDC-based approximation $\phiApproxCoarse$
of the $\phiCoarse$ coarse time integrator.
Combined with the prolongation on the fine level~\eqref{eq:PFASST-prolong}, we get the modified CGC update
\begin{equation}\label{eq:globalCoarseBGS}
    \uvect^{k+1} = \uvect^{k+1/2} + \TCtoFBar \MGSTilde^{-1} \TFtoCBar
    (\vect{f}-\AMat\uvect^{k+1/2}), \quad
    \MGSTilde =
    \begin{bmatrix}
        \phiApproxCoarse & & \\
        - \chiCoarse & \phiApproxCoarse & \\
        & \ddots & \ddots
    \end{bmatrix},
\end{equation}
and together with \eqref{PFASSTSmoothingStep} a two level method
for the global system~\eqref{eq:globalProblem}~\cite[Sec.~2.2]{bolten2018asymptotic}.
Note that this is the same iteration we obtained for the CGC in Section~\ref{sec:coarseGridCorrection}, except that the coarse operator $\phiCoarse$ has been replaced by $\phiApproxCoarse$.
Assumption~\ref{ass:deltaChi} holds, since using Lobatto or Radau-II nodes means $\Hmat$ has the form~\eqref{eq:collocation}, which implies
\begin{equation}
    \Delta_\chi = \TFtoC\Hmat-\HTilde\TFtoC = 0.
\end{equation}
Using similar computations as in Section~\ref{sec:coarseGridCorrection}
and the block operators defined for
collocation and SDC (\cf Section~\ref{ex:collocation} and
Section~\ref{sec:SDC}) we obtain the block iteration
\begin{equation}
    \uvect_{n+1}^{k+1} =
    [\Imat-\TCtoF	(\Imat - \QDeltaTilde)^{-1}\TFtoC(\Imat - \Qmat)]\uvect^{k+1/2}_{n+1}
    + \TCtoF(\Imat - \QDeltaTilde)^{-1}\TFtoC\Hmat\uvect_{n}^{k+1}
\end{equation}
by substitution into \eqref{eq:twoLevelSimple}.
Finally, the combination of the two gives
\begin{equation}\label{eq:PFASSTBlockIteration}
    \begin{split}
        \uvect_{n+1}^{k+1} &=
        [\Imat-\TCtoF	(\Imat - \QDeltaTilde)^{-1}\TFtoC(\Imat - \Qmat)]
        [\Imat-\QDelta]^{-1}(\Qmat-\QDelta)\uvect_{n+1}^k \\
        &~+ (\Imat-\TCtoF	[\Imat-\QDeltaTilde]^{-1}\TFtoC(\Imat-\Qmat))
        [\Imat-\QDelta]^{-1}\Hmat\uvect_{n}^k \\
        &~+ \TCtoF(\Imat - \QDeltaTilde)^{-1}\TFtoC\Hmat\uvect_{n}^{k+1}.
    \end{split}
\end{equation}
Using the generic formulation with the $\phiOp$ operators gives\footnote{
    We implicitly use
        $[\Imat-\QDelta]^{-1}(\Qmat-\QDelta)=\Imat - [\Imat - \QDelta]^{-1}(\Imat-\Qmat) = \eyeMat - \phiApprox^{-1} \phiOp$,
        see \eqref{eq:blockCoefficientsBGSSDC}.
}
\begin{equation}
    \begin{split}
        \uvect_{n+1}^{k+1} &=
        [\eyeMat - \TCtoF\phiApproxCoarse^{-1}\TFtoC \phiOp]
        (\eyeMat - \phiApprox^{-1} \phiOp)\uvect_{n+1}^k \\
        &~+ (\eyeMat - \TCtoF\phiApproxCoarse^{-1}\TFtoC \phiOp)
        \phiApprox^{-1}\chiOp\uvect_{n}^k
        + \TCtoF\phiApproxCoarse^{-1}\TFtoC\chiOp\uvect_{n}^{k+1}.
    \end{split}
\end{equation}
This is again a primary block iteration in the sense of Definition~\ref{def:primaryBlockIteration}, but in contrast to most previously described block iterations, all block operators are non-zero.

\subsection{Similarities between \PFASST, \TMG and \Parareal}\label{sec:PFASST-comp}

From the description in the previous section, it is clear that
\PFASST is very similar to \TMG.
While \TMG uses a (damped) block Jacobi smoother for pre-relaxation and
a direct solve for the CGC, \PFASST uses instead an approximate Block Jacobi
smoother, and solves the CGC using one (or more) ABGS iterations
on the coarse grid.
\begin{table}
    \renewcommand{\arraystretch}{1.5}\centering\small
    \begin{tabular}{c|c|c}
        \backslashbox{CGC}{Smoother}
        & Block Jacobi ($\omega=1$)
        & Approximate Block Jacobi \\\hline
        Direct Solver
        & \TMG ($\omega=1$)& \TMGFine
        \\
        ABGS (one step) & \TMGCoarse & Two-level \PFASST
    \end{tabular}\vspace{5pt}
    \caption{Classification of two-level \TMG methods, depending on their smoother for fine-level relaxation and computation of the Coarse Grid Correction (CGC).}
    \label{tab:twoGridClassification}
\end{table}
This interpretation was obtained by Bolten et al.~\cite[Theorem 1]{bolten2017multigrid},
but is derived here using the GFM framework,
and we summarize those differences in Table~\ref{tab:twoGridClassification}.
Changing only the CGC or the smoother in \TMG with $\omega=1$ in contrast to both like in PFASST
produces two further PinT algorithms.
We call those $\TMGCoarse$ (replacing the coarse solver by one step of ABGS)
and $\TMGFine$ (replacing the fine Block Jacobi solver by ABJ).
Note that \TMGCoarse can be interpreted as
\Parareal using an approximate integration operator and larger time step for
the coarse propagator if we set
\begin{equation}
    \mathcal{G} := \TCtoF\phiApproxCoarse^{-1}\TFtoC\chiOp.
\end{equation}
Thus, the version of \Parareal used in Section~\ref{sec:analysisParareal} is equivalent
to \TMGCoarse, and differs from \PFASST only by the type of
smoother used on the fine level.

\subsection{Analysis and numerical experiments}

\subsubsection{Convergence of \PFASST iteration components}
\label{sec:analysisBlockSDC}

Since Block Jacobi SDC~\eqref{eq:blockIterationBJ} can be written as a primary block iteration, we
can apply Theorem~\ref{th:errBoundPBI} with $\beta=0$ to get the error bound
\begin{equation}
  e_{n+1}^{k} \leq \begin{cases}
    \delta (\gamma + \alpha)^k \text{ if } k \leq n \\
    \displaystyle \delta \gamma^k \sum_{i=0}^{n}
     \binom{k}{i}\left(\frac{\alpha}{\gamma}\right)^i \text{ otherwise,}
  \end{cases}
\end{equation}
with $\gamma:= \norm{[\Imat - \QDelta]^{-1}(\Qmat-\QDelta)}$,
$\alpha:= \norm{[\Imat - \QDelta]^{-1}\Hmat}$.
Note that $\gamma$ is proportional to $\lambda\dt$
through the $\Qmat-\QDelta$ term and for small $\dt$, $\alpha$ tends to $\norm{\Hmat}$ which is constant.
We can identify two convergence regimes:  for early iterations ($k\leq n$), the
bound does
not contract if $\gamma+\alpha \geq 1$ (which is generally the case).
For later iterations ($k>n$), a small-enough time step leads to convergence of the algorithm through the $\gamma^k$ factor.

Similarly, for Block Gauss-Seidel SDC \eqref{eq:blockIterationBGS}, Theorem~\ref{th:errBoundPBI} with $\alpha=0$ gives
\begin{equation}
  e^k_{n+1} \leq \delta
  \frac{\gamma^k}{(k-1)!} \sum_{i=0}^{n}\prod_{l=1}^{k-1}(i+l)\beta^{i},
\end{equation}
where $\gamma := \norm{[\Imat - \QDelta]^{-1}(\Qmat-\QDelta)}$, $\beta:= \norm{[\Imat - \QDelta]^{-1}\Hmat}$.
This iteration contracts already in early iterations if $\gamma$ is small enough.
Since the value for $\gamma$ is the same for both Block Gauss-Seidel SDC
and Block Jacobi-SDC, both algorithms have an asymptotically similar
convergence rate.

We illustrate this with the following example.
Let $\lambda:=i$, $u_0:=1$, and let the time interval $[0,\pi]$ be divided into $N=10$ sub-intervals.
Inside each sub-interval, we use one step of the collocation method from
Section~\ref{ex:collocation} with $M:=10$
Lobatto-Legendre nodes~\cite{gautschi2004orthogonal}.
This gives us block variables of size $M=10$ and we choose $\QDelta$ as the
matrix defined by a single Backward Euler step between nodes to
build the $\phiApprox$ operator.
The starting value $\uvect^0$ for the iteration is initialized with random numbers starting from the same seed.
Figure~\ref{fig:blockSDC} (left) shows the numerical error for the last block using the $L^{\infty}$ norm, the bound obtained with the GFM method and the linear bound using the norm of the global iteration matrix.
\begin{figure}\label{fig:blockSDC}
  \centering
  \includegraphics[width=0.49\linewidth]{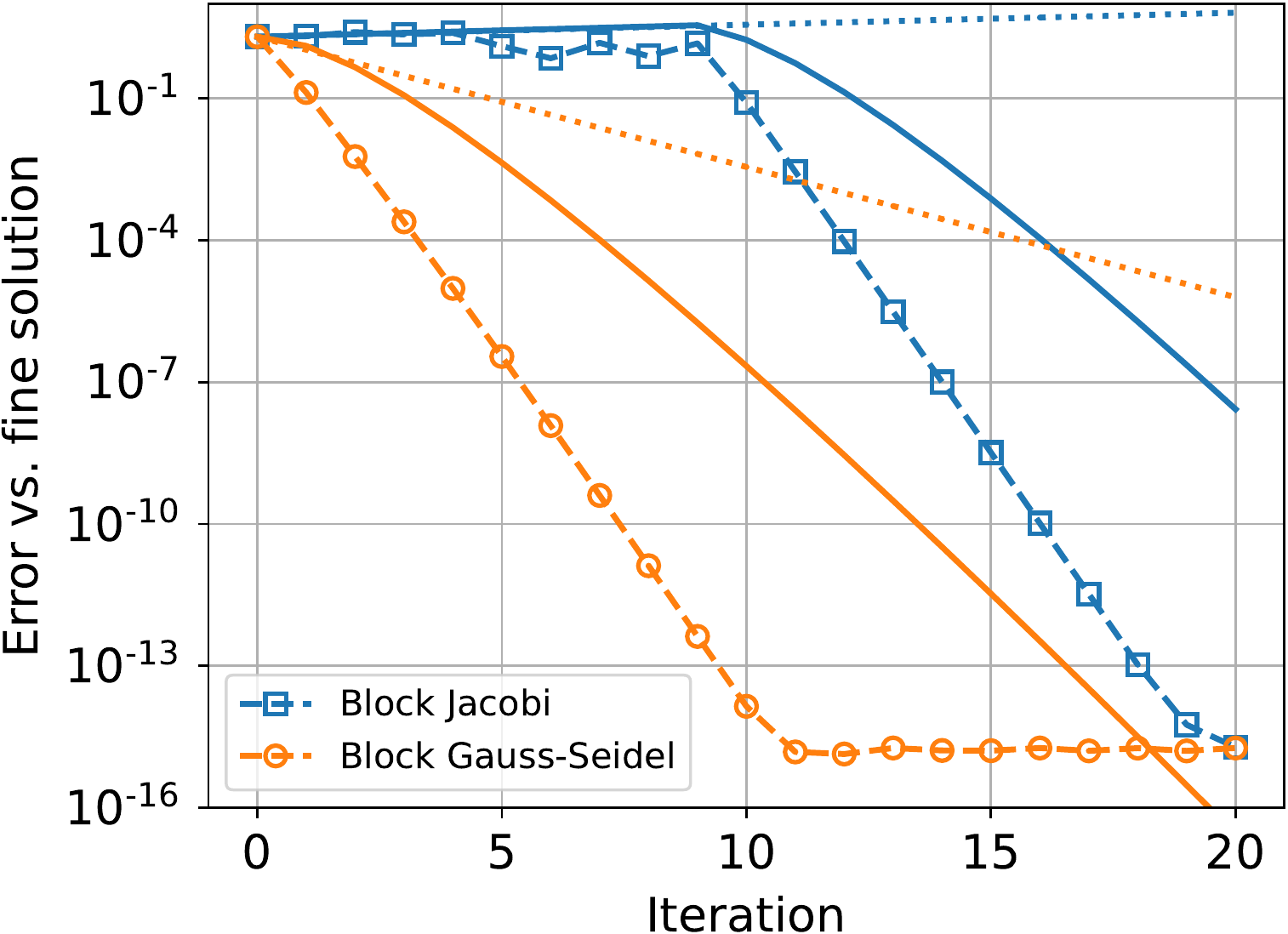}%
  \includegraphics[width=0.49\linewidth]{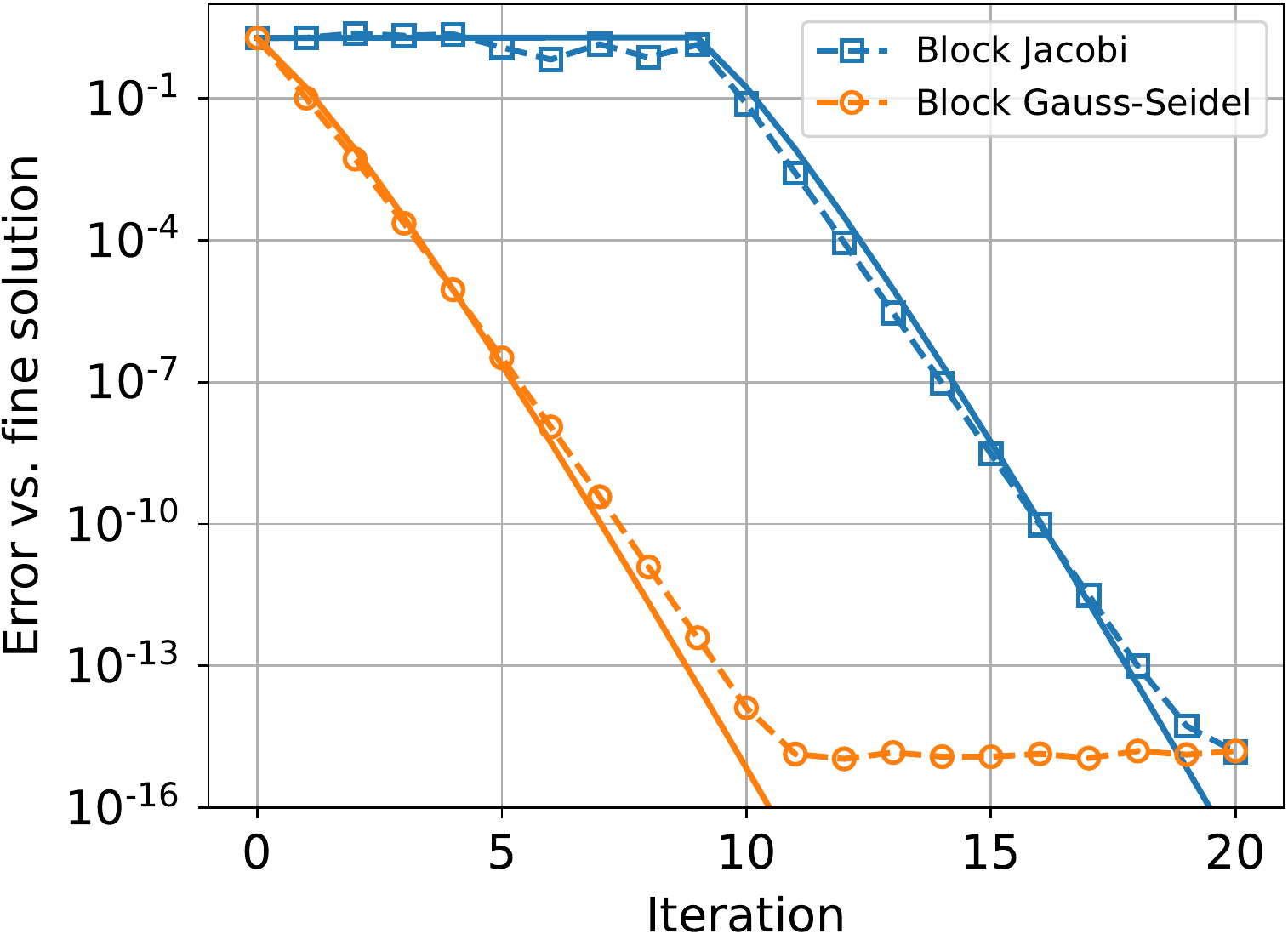}%
  \caption{Comparison of numerical errors with GFM-bounds for Block
    Jacobi SDC and Block Gauss-Seidel SDC. Left: error on the block variables
    (dashed), GFM-bounds (solid), linear
    bound from the iteration matrix (dotted).  Right: error estimate
    using the interface approximation from Corollary~\ref{cor:interfaceApproximation}.
    Note that the numerical errors on block variables (left) and at the interface
    (right) are close but not identical (see Remark~\ref{rem:interfaceVSblock}).}
\end{figure}
As for \Parareal in Section~\ref{sec:analysisParareal}, the GFM-bound is similiar to
the iteration matrix bound for the first few iterations, but much tighter for later iterations.
In particular, the linear bound cannot show the change in convergence regime of
the Block Jacobi SDC iteration (after $k=10$) but the GFM-bound does.
Also, we observe that while the GFM-bound overestimates the error,
the interface approximation of Corollary~\ref{cor:interfaceApproximation}
gives a very good estimate of the error at the interface, see
Figure~\ref{fig:blockSDC} (right).

\subsubsection{Analysis and convergence of \PFASST}
\label{sec:PFASST-analysis}

The GFM framework provides directly an error bound for \PFASST:
applying Theorem~\ref{th:errBoundPBI}
to~\eqref{eq:PFASSTBlockIteration} gives
\begin{equation}
  e_{n+1}^k \leq \delta \gamma^k \sum_{i=0}^{\min(n, k)} \sum_{l=0}^{n-i}
  \binom{k}{i}\binom{l+k-1}{l}
  \left(\frac{\alpha}{\gamma}\right)^i\beta^l,
\end{equation}
with $\gamma:= ||
[\Imat-\TCtoF	(\Imat - \QDeltaTilde)^{-1}\TFtoC(\Imat - \Qmat)]
[\Imat-\QDelta]^{-1}(\Qmat-\QDelta)||$,
$\beta := ||
\TCtoF(\Imat - \QDeltaTilde)^{-1}\HTilde\TFtoC ||$,
and $\alpha := ||
(\Imat-\TCtoF [\Imat-\QDeltaTilde]^{-1}\TFtoC(\Imat-\Qmat))
[\Imat-\QDelta]^{-1}\Hmat ||$.

\begin{figure}\label{fig:PFASST}
    \centering
    \includegraphics[width=0.49\linewidth]{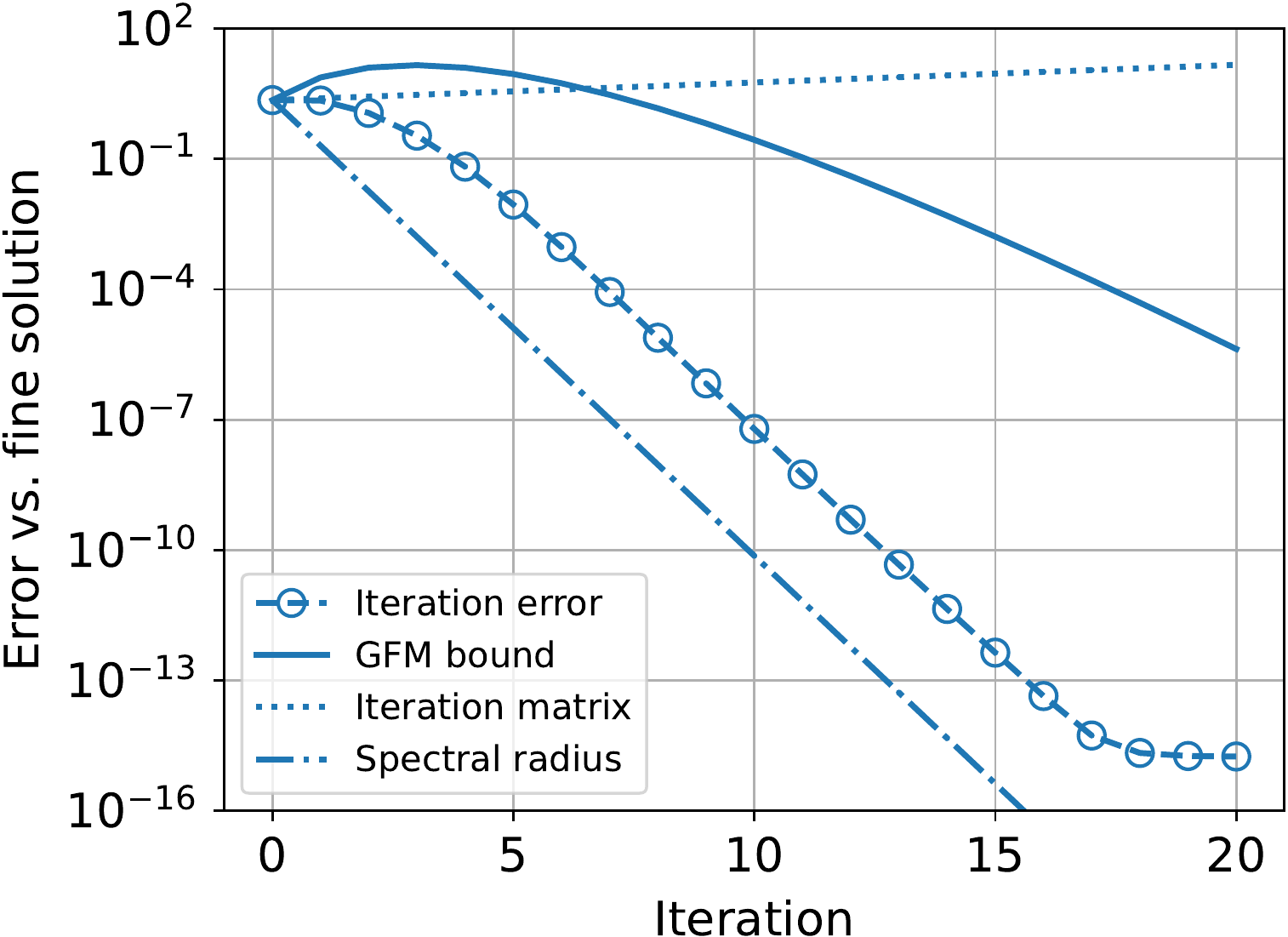}%
    \includegraphics[width=0.49\linewidth]{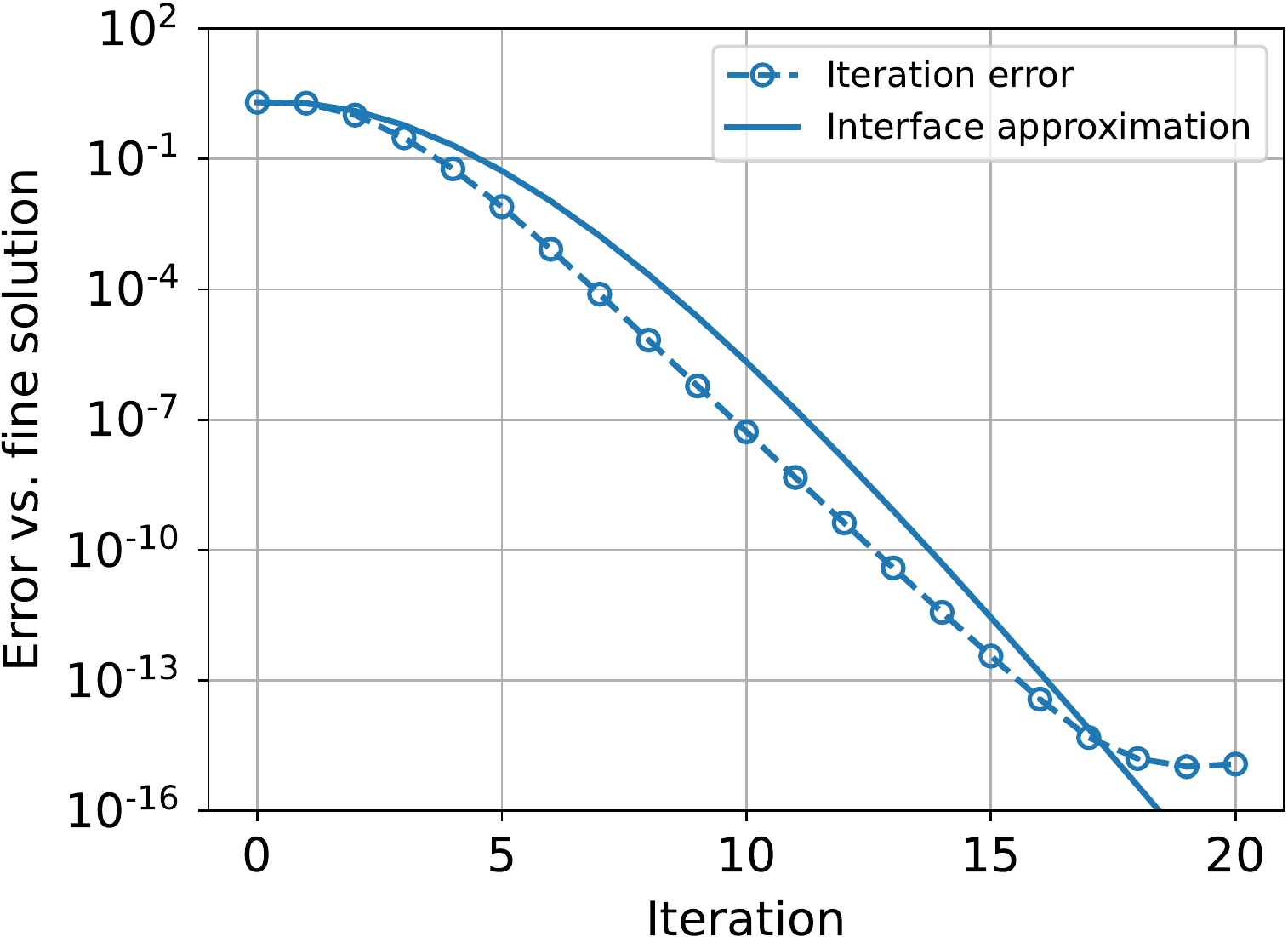}%
    \caption{Comparison of numerical errors with GFM-bounds for \PFASST.
        Left: error bound using volume values.  Right: estimate using
        the interface approximation.
        Note that the numerical errors on block variables (left) and at the interface
            (right) are very close but not identical (see Remark~\ref{rem:interfaceVSblock}).}
\end{figure}
We compare this bound with numerical experiments.
Let $\lambda:=i$, $u_0:=1$.
The time interval $[0,2\pi]$ for the Dahlquist problem \eqref{eq:dahlquist} is
divided into $N=10$ sub-intervals.
Inside each sub-interval we use $M:=6$ Lobatto-Legendre nodes on the fine level
and $\MCoarse:=2$ Lobatto nodes on the coarse level.
The $\QDelta$ and $\QDeltaTilde$ operators use Backward Euler.
In Figure~\ref{fig:PFASST} (left) we compare the measured numerical error with the
GFM-bound and the linear bound from the iteration
matrix.
As in Section~\ref{sec:analysisBlockSDC}, both bounds overestimate the
numerical error, even if the GFM-bound shows convergence for the later
iterations, which the linear bound from the iteration matrix cannot.
We also added an error estimate built using the spectral radius of the
iteration matrix, for which an upper bound was derived in \cite{bolten2018asymptotic}.
For this example, the spectral radius reflects the asymptotic convergence rate for the last iterations better than GFM.
This highlights a weakness of the current GFM-bound: applying norm and triangle inequalities to the vector error recurrence~\eqref{eq:errRecurrence} can induce a large approximation error in the scalar error recurrence~\eqref{eq:errRecurrenceScalar} that is then solved with generating functions.
Improving this is planned for future work.

However, one advantage of the GFM-bound over the spectral radius is its
generic aspect allowing it to be applied to many iterative algorithms,
even those having an iteration matrix with spectral radius equal to zero
like \Parareal~\cite{ruprecht2012explicit}.
Furthermore, the interface approximation from
Corollary~\ref{cor:interfaceApproximation} allows us to get a significantly
better estimation of the numerical error, as
shown in Figure~\ref{fig:PFASST} (right).
For the GFM-bound we have $(\alpha,\beta, \gamma) = (0.16, 1, 0.19)$,
while for the interface approximation we get
$(\bar{\alpha}, \bar{\beta}, \bar{\gamma}) = (0.16, 0.84, 0.02)$.
In the second case, since $\bar{\gamma}$ is one order smaller than the other
coefficients, we get an error estimate that is closer to the one for \Parareal
in Section~\ref{sec:analysisParareal} where $\gamma=0$.
This similarity between \PFASST and \Parareal ~(\cf Section~\ref{sec:PFASST-comp})
will be highlighted in the next section.
\section{Comparison of iterative PinT algorithms}
\label{sec:comparison}

Using the notation of the GFM framework, we provide the
primary block iterations of all iterative PinT algorithm
investigated throughout this paper in Table~\ref{tab:blockIterations}.
\begin{table}
    \renewcommand{\arraystretch}{1.8}\centering\scriptsize
    \begin{tabular}{c|c|c|c}
        Algorithm
        & $\BMat_1^0$ ($\uvect_{n+1}^k$)
        & $\BMat_0^0$ ($\uvect_{n}^k$)
        & $\BMat_0^1$ ($\uvect_{n}^{k+1}$) \\
        \hline
        damped Block Jacobi
        & $\eyeMat-\omega\eyeMat$
        & $\omega\phiOp^{-1}\chiOp$
        &  -- \\
        ABJ
        & $\eyeMat - \phiApprox^{-1} \phiOp$
        & $\phiApprox^{-1}\chiOp$
        & -- \\
        ABGS
        & $\eyeMat - \phiApprox^{-1} \phiOp$
        & --
        & $\phiApprox^{-1}\chiOp$\\
        \hline
        \Parareal
        & --
        & $(\phiOp^{-1}-\phiApprox^{-1})\chiOp$
        & $\phiApprox^{-1}\chiOp$\\
        \TMG
        & $(1-\omega)(\eyeMat - \TCtoF\phiCoarse^{-1}\TFtoC \phiOp)$
        & $\omega(\phiOp^{-1}-\TCtoF\phiCoarse^{-1}\TFtoC)\chiOp$
        & $\TCtoF\phiCoarse^{-1}\TFtoC\chiOp$ \\
        \TMGCoarse
        & --
        & $(\phiOp^{-1}-\TCtoF\phiApproxCoarse^{-1}\TFtoC)\chiOp$
        & $\TCtoF\phiApproxCoarse^{-1}\TFtoC\chiOp$ \\
        \TMGFine
        & $(\eyeMat - \TCtoF\phiCoarse^{-1}\TFtoC \phiOp)
        (\eyeMat - \phiApprox^{-1} \phiOp)$
        & $(\phiApprox^{-1}
        - \TCtoF\phiCoarse^{-1}\TFtoC \phiOp\phiApprox^{-1})
        \chiOp$
        & $\TCtoF\phiCoarse^{-1}\TFtoC\chiOp$\\
        \PFASST
        & $(\eyeMat - \TCtoF\phiApproxCoarse^{-1}\TFtoC \phiOp)
        (\eyeMat - \phiApprox^{-1} \phiOp)$
        & $(\phiApprox^{-1}
        - \TCtoF\phiApproxCoarse^{-1}\TFtoC \phiOp\phiApprox^{-1})
        \chiOp$
        & $\TCtoF\phiApproxCoarse^{-1}\TFtoC\chiOp$
    \end{tabular}\vspace{5pt}
    \label{tab:blockIterations}
    \caption{ Summary of all the methods we analyzed, and their block
      iteration operators.  Note that TMG with $\omega=1$ and
      \TMGCoarse corresponds to \Parareal with a specific
      choice of the coarse propagator.}
\end{table}
In particular, the first rows summarize the basic block iterations
used as components to build the iterative PinT methods.
While damped Block Jacobi (Section~\ref{ex:blockJacobi}) and ABJ
(Section~\ref{sec:BJ-SDC}) are more suitable for smoothing\footnote{
    Note that algorithms used as smoother have $\BMat_0^1=0$, which is a
        necessary condition for parallel computation across all blocks.
},
ABGS (Section~\ref{ex:ABGS}) is mostly used as solver
(\eg to compute the CGC).
This allows us to compare the convergence of each block
iteration, and we illustrate this with the following examples.

\begin{table}\renewcommand{\arraystretch}{1.5}\centering
    \begin{tabular}{c|c||c|c|c}
        & $\phiOp^{-1}\chiOp$ & $\phiApprox^{-1}\chiOp$ &
        $\TCtoF\phiCoarse^{-1}\TFtoC\chiOp$ &
        $\TCtoF\phiApproxCoarse^{-1}\TFtoC\chiOp$ \\ \hline
        Figure~\ref{fig:comparison} (left)
        & $1.20e^{-5}$ & $3.57e^{-1}$ & $1.19e^{-2}$ & $4.87e^{-1}$ \\
        Figure~\ref{fig:comparison} (right)
        & $3.14e^{-4}$ & $6.24e^{-2}$ & $5.14e^{-3}$ & $2.67e^{-1}$ \\
    \end{tabular}\vspace{5pt}
    \caption{Maximum error over time for each block propagator run sequentially.
        The first column shows the error of the fine propagator,
            while the next three columns show the error of the three possible approximate
            propagators.
            In the top row, $\phiOp$ corresponds to a collocation method with $M=5$
            nodes while $\phiCoarse$ is a collocation method with $M=3$ nodes.
            $\phiApprox$ is a backward Euler method with $M=5$ steps per block
            while $\phiApproxCoarse$ is backward Euler with $M=3$ steps per
            block.
            In the bottom row, $\phiOp$ corresponds to $M=5$ uniform steps per block of
            a $4^{th}$ order Runge-Kutta method, $\phiCoarse$ is the same method
            with $M=3$ steps per block.
            $\phiApprox$ is a $2^{nd}$ order Runge-Kutta method (Heun) with $M=5$
            uniform steps per block while $\phiApproxCoarse$ is the same method
            with $M=3$ uniform time steps per block.}
    \label{tab:discretizationErrors}
\end{table}
We consider the Dahlquist problem with $\lambda:=2i-0.2$, $u_0=1$.
First, we decompose the simulation interval $[0,  2\pi]$ into $N=10$
sub-intervals.
Next, we choose a block discretization with $M=5$ Lobatto-Legendre nodes,
a collocation method on each block for fine integrator $\phiOp$,
see Section~\ref{ex:collocation}.
We build a coarse block discretization using $\MCoarse=3$, and define
on each level an approximate integrator using Backward Euler.
This allows us to define the $\phiApprox$, $\phiCoarse$ and
$\phiApproxCoarse$ integrators, see the legend of
Table~\ref{tab:discretizationErrors} for more details, where
 we show the maximum absolute error in time for each of the four propagators run sequentially.
The high order collocation method with $M=5$ nodes $\phiOp^{-1} \chiOp$ is the most accurate.
The coarse collocation method with $M=3$ nodes interpolated to the fine mesh is still more accurate than the backward Euler method with $M=5$ nodes $\phiApprox^{-1} \chiOp$ or the backward Euler method with $M=3$ interpolated to the fine mesh.
Then we run all algorithms in Table~\ref{tab:blockIterations},
initializing the block variable iterate with the same random initial guess.
The error for the last block variable with respect to the fine
sequential solution is shown in Figure~\ref{fig:comparison} (left).
In addition, we show the same results in Figure~\ref{fig:comparison} (right),
but using the classical $4^{th}$ order Runge-Kutta method as fine propagator,
$2^{nd}$ order Runge-Kutta (Heun method) for the approximate integration operator
and equidistant points using a volume formulation as described in
Section~\ref{ex:RungeKutta}.
\begin{figure}\label{fig:comparison}
	\centering
	\includegraphics[width=0.49\linewidth]{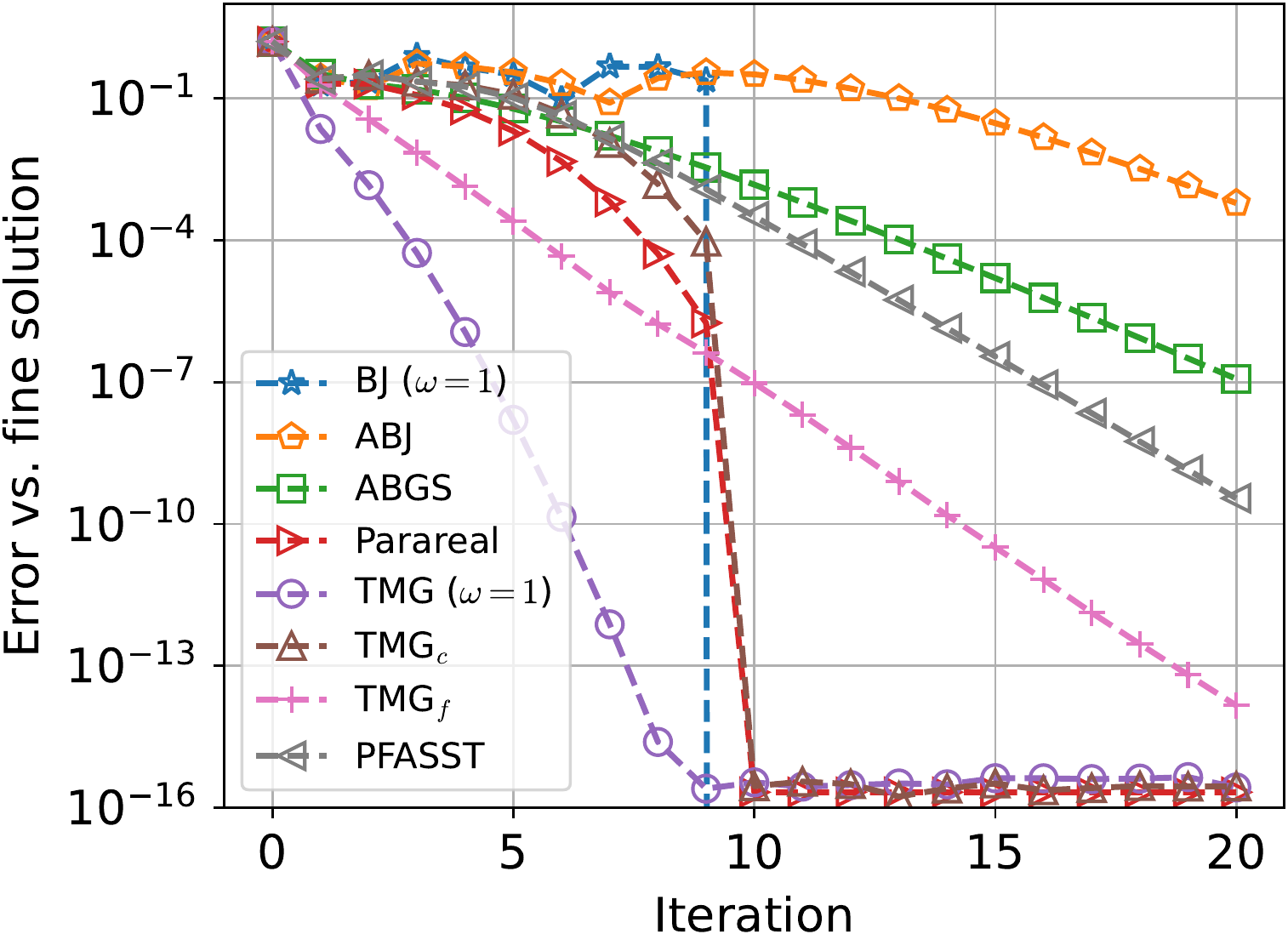}~%
	\includegraphics[width=0.49\linewidth]{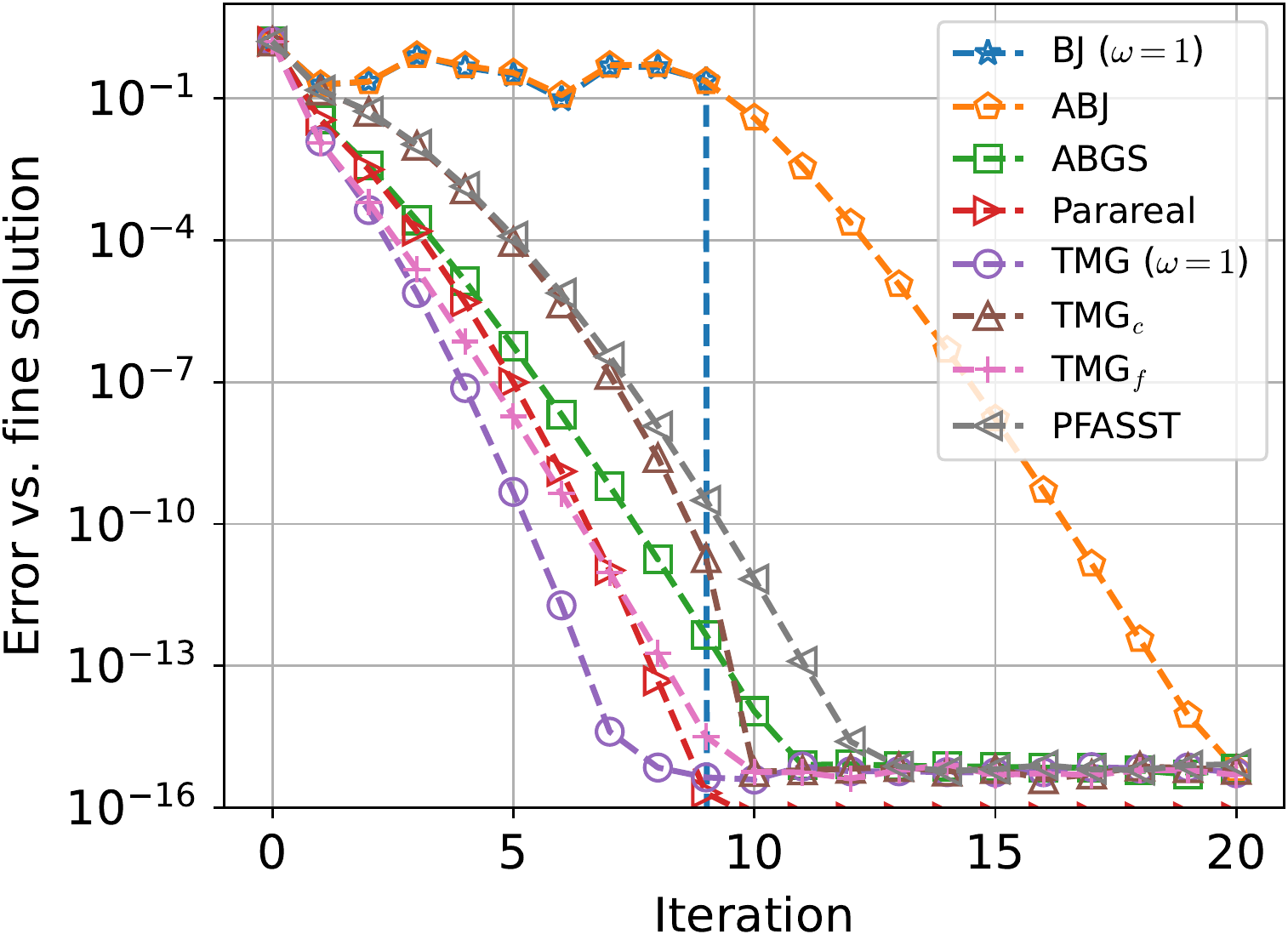}%
	\caption{Comparison of iterative methods convergence using the GFM framework.
	 Left: collocation as fine integrator.
	 Right: $4^{th}$ order Runge-Kutta method as fine integrator.
 	 \Parareal ($_{\omega=1}$) and \Parareal (\TMGCoarse) denote
 	 a specific coarse propagator for \Parareal.}
\end{figure}
Note that \Parareal, \TMG$_{\omega=1}$ and \TMGCoarse
are each \Parareal algorithms using respectively
$\phiApprox^{-1}\chiOp$, $\TCtoF\phiCoarse^{-1}\TFtoC\chiOp$
and $\TCtoF\phiApproxCoarse^{-1}\TFtoC\chiOp$ as coarse propagator $\mathcal{G}$
(see Table~\ref{tab:discretizationErrors} for their discretization error).

The TMG iteration converges fastest, since it uses the most accurate block integrators on both levels, cf. Table~\ref{tab:discretizationErrors}.
Keeping the same CGC but approximating the smoother, \TMGFine improves the first iterations, but convergence for later iterations is closer to \PFASST.
This suggests that convergence for later iterations is mostly governed by the accuracy of the smoother since both \TMGFine and \PFASST use ABJ.
This is corroborated by the comparison of \PFASST and \TMGCoarse, which differ only in their
choice of smoother.
While the exact Block Jacobi relaxation makes \TMGCoarse converge after $k=N$ iterations (a well known property of \Parareal), using the ABJ smoother means that \PFASST does not share this property.

On the other hand, the first iterations are also influenced by the CGC accuracy.
The iteration error is very similar for \PFASST and \TMGCoarse which have the same CGC.
This is more pronounced when using the $4^{th}$ order Runge-Kutta method for $\phiOp$, as we see in Figure~\ref{fig:comparison} (right).
Early iteration errors are similar for two-level methods that use the same CGC (\TMG / \TMGFine, and \PFASST / \TMGCoarse).
Similarities of the first iteration errors can also be observed
for \Parareal and ABGS.
Both algorithms use the same $\BMat_0^1$ operator, see
Table~\ref{tab:blockIterations}.
This suggests that early iteration errors are mostly governed by the
accuracy of $\BMat_0^1$, which is corroborated by the two-level methods
(\TMG and \TMGFine use the same $\BMat_0^1$ operator,
as \PFASST and \TMGCoarse).
\begin{remark}
  An important aspect of this analysis is that it compares
    \emph{only the convergence} of each algorithm, and not
      their overall computational cost.
  For instance, \PFASST and \TMGCoarse appear to
  be equivalent for the first iterations,
  but the block iteration of \PFASST is cheaper than \TMGCoarse, because
  an approximate block integrator is used for relaxation.
  To account for this and build a model for computational efficiency, the GFM framework would need to be combined with a model for computational cost of the different parts in the block iterations.
  Such a study is beyond the scope of this paper but is the subject of ongoing work.
\end{remark}

\section{Conclusion}\label{sec:conclusion}

We have shown that the Generating Function Method (GFM) can be used to compare convergence of different iterative PinT algorithms.
To do so, we formulated popular methods like \Parareal, \PFASST, \MGRIT or \TMG in a common framework based on the definition of a primary block iteration.
The GFM analysis showed that all these methods eventually converge super-linearly\footnote{
    This is due to the factorial term stemming from the binomial sums in the estimates~\eqref{Bound1}-\eqref{Bound4}.
    }
due to the evolution nature of the problems.
We confirmed this by numerical experiments and our \textsc{Python} code is publically available \cite{lunet2023gfm}.

Our analysis opens up further research directions.
For example, studying multi-step block iterations like \MGRIT with FCF-relaxation and more complex two-level methods without Assumption~\ref{ass:deltaChi} would be a useful extension of the GFM framework.
Similarly, an extension to multi-level versions of \STMG, \PFASST and \MGRIT would be very
valuable.
Finally, in practice PinT methods are used to solve space-time problems.
The GFM framework should be able to provide convergence bounds in this case as
well, potentially even for non-linear problems, considering GFM was used
successfully to study \Parareal applied to non-linear systems of
ODEs~\cite{gander2008nonlinear}.

\section*{Acknowledgments}
We greatly appreciate the very detailed feedback from the anonymous reviewers.
It helped a lot to improve the organization of the paper and to make it more
accessible.

\appendix

\section{Error bounds for \emph{Primary Block Iterations}}
\label{ap:errPrimary}

\subsection{Incomplete Primary Block Iterations}
First, we consider
\begin{align}
	\text{(PBI-1)} :\quad \uvect^{k+1}_{n+1} &=
		\BMat^1_0\left(\uvect^{k+1}_{n}\right) + \BMat^0_0\left(\uvect^{k}_{n}\right),
	\label{eq:pbi1}\\
	\text{(PBI-2)} :\quad \uvect^{k+1}_{n+1} &=
		\BMat^0_1(\uvect^k_{n+1}) + \BMat^0_0\left(\uvect^{k}_{n}\right),
	\label{eq:pbi2}\\
	\text{(PBI-3)} :\quad \uvect^{k+1}_{n+1} &=
		\BMat^0_1(\uvect^k_{n+1}) + \BMat^1_0\left(\uvect^{k+1}_{n}\right),\label{eq:pbi3}
\end{align}
where one block operator is zero. (PBI-1) corresponds to \Parareal, (PBI-2)
to Block Jacobi SDC and (PBI-3) to Block Gauss-Seidel SDC.
We recall the notations :
\begin{equation}\label{alphabetagammadef}
\alpha := \norm{\BMat^0_0}, \quad \beta:= \norm{\BMat^1_0}, \quad
\gamma := \norm{\BMat^0_1}.
\end{equation}
Application of Lemma~\ref{lem:primary} gives the recurrence relations
\begin{align}
\text{(PBI-1)} :\quad \rho_{k+1}(\zeta)
	&\leq \frac{\alpha \zeta}{1-\beta\zeta} \rho_{k}(\zeta)
	\Longrightarrow \rho_{k}(\zeta)
	\leq \alpha^{k}\left(\frac{\zeta}{1-\beta\zeta}\right)^{k}\rho_{0}(\zeta)\\
\text{(PBI-2)} :\quad \rho_{k+1}(\zeta) &\leq (\gamma + \alpha\zeta)
	\rho_{k}(\zeta)
	\Longrightarrow \rho_{k}(\zeta)
	\leq \gamma^{k}\left(1 + \frac{\alpha}{\gamma}\zeta\right)^{k}
	\rho_{0}(\zeta)\\
\text{(PBI-3)} :\quad \rho_{k+1}(\zeta) &\leq \frac{\gamma}{1-\beta\zeta}
	\rho_{k}(\zeta)
	\Longrightarrow \rho_{k}(\zeta)
	\leq \gamma^{k}\frac{1}{(1-\beta\zeta)^{k}} \rho_{0}(\zeta)
\end{align}
for the corresponding generating functions.
Using definition\footnote{
    The definition of $\delta$ as maximum error for $n\in\{0,\dots, N\}$ can be extended to $n\in\mathbb{N}$, as the error values for $n>N$ do not matter and can be set to zero.}~\eqref{eq:deltaDefinition} for $\delta$, we find that
$
\rho_{0}(\zeta) \leq \delta \sum_{n=0}^{\infty}\zeta^{n+1}.
$
By using the binomial series expansion
\begin{equation}
\frac{1}{(1-\beta\zeta)^{k}}
= \sum_{n=0}^{\infty} \binom{n+k-1}{n}(\beta\zeta)^n
\end{equation}
for $k>0$ and the Newton binomial sum, we obtain for the three block iterations
\begin{align}
\text{(PBI-1)} :\quad \rho_{k}(\zeta) &\leq \delta \alpha^{k} \zeta
	\left[\sum_{n=0}^{\infty} \binom{n+k-1}{n}\beta^n\zeta^{n+k}\right]
	\left[\sum_{n=0}^{\infty}\zeta^{n}\right] \\
\text{(PBI-2)} :\quad \rho_{k}(\zeta) &\leq \delta \gamma^{k} \zeta
	\left[\sum_{n=0}^{k}
	\binom{k}{n}\left(\frac{\alpha}{\gamma}\right)^n\zeta^{n}\right]
	\left[\sum_{n=0}^{\infty}\zeta^{n}\right] \\
\text{(PBI-3)} :\quad \rho_{k}(\zeta) &\leq \delta \gamma^{k} \zeta
	\left[\sum_{n=0}^{\infty} \binom{n+k-1}{n}\beta^n\zeta^{n}\right]
	\left[\sum_{n=0}^{\infty}\zeta^{n}\right].
\end{align}

\paragraph{Error bound for PBI-1}
We simplify the expression using
\begin{equation}
\left[\sum_{n=0}^{\infty} \binom{n+k-1}{n}\beta^n\zeta^{n+k}\right] =
\left[\sum_{n=k}^{\infty} \binom{n-1}{n-k}\beta^{n-k}\zeta^{n}\right],
\end{equation}
and then the series product formula
\begin{equation}\label{eq:seriesProduct}
\left[\sum_{n=0}^{\infty} a_n\zeta^n\right] \left[\sum_{n=0}^{\infty} b_n\zeta^n\right]
	= \sum_{n=0}^{\infty} c_n\zeta^n, \quad c_n = \sum_{i=0}^{n} a_i b_{n-i},
\end{equation}
with $b_n = 1$ and
\begin{align}
a_n = \begin{cases}
	0 \text{ if } n < k, \\
	\displaystyle\binom{n-1}{n-k}\beta^{n-k} \text{ otherwise.}
\end{cases}
\end{align}
From this we get
\begin{equation}
	c_n = \sum_{i=k}^{n} \binom{i-1}{i-k}\beta^{i-k}
	= \sum_{i=0}^{n-k} \binom{i+k-1}{i}\beta^{i}
	= \sum_{i=0}^{n-k} \frac{\prod_{l=1}^{k-1}(i+l)}{(k-1)!}\beta^i,
\end{equation}
using the convention that the product reduces to one when there are no terms in it.
Identifying coefficients in the power series and rearranging terms yields for $k>0$
\begin{equation}\label{PBI-1bound}
\boxed{\text{(PBI-1)} :\quad
	e_{n+1}^{k} \leq
	\delta \frac{\alpha^k}{(k-1)!} \sum_{i=0}^{n-k}\prod_{l=1}^{k-1}(i+l)\beta^{i}.}
\end{equation}
Following an idea by Gander and Hairer~\cite{gander2008nonlinear}, we
can also consider the error recurrence
$e_{n+1}^{k+1} \leq \alpha e_{n}^k + \bar{\beta} e_{n}^{k+1}$, $
\bar{\beta} = \max(1, \beta)$.
Using the upper bound
$
	\sum_{n=0}^{\infty}\zeta^n = \frac{1}{1-\zeta} \leq \frac{1}{1-\bar{\beta}\zeta},
$
for the initial error, we avoid the series product and get
$
	\rho_{k}(\zeta) \leq \delta\alpha^k \frac{\zeta^k}{(1-\bar{\beta})^{k+1}}
$
as bound on the generating function.
We then obtain the simpler error bound
\begin{equation}
	e_{n+1}^k \leq \delta \frac{\alpha^k}{k!}
	\bar{\beta}^{n-k}\prod_{l=1}^{k}(n+1-l)
\end{equation}
as in the proof of \cite[Th.~1]{gander2008nonlinear}.
\paragraph{Error bound for PBI-2}
We use~\eqref{eq:seriesProduct} again with $b_n = 1$ to get
\begin{align}
a_n = \begin{cases}
\displaystyle\binom{k}{n}\left(\frac{\alpha}{\gamma}\right)^n \text{ if } n
\leq k, \\
0 \text{ otherwise.}
\end{cases}
\end{align}
From this we get
$
c_n = \sum_{i=0}^{\min(n,k)}\binom{k}{i}\left(\frac{\alpha}{\gamma}\right)^i,
$
which yields for $k>0$ the error bound
\begin{equation}
\boxed{\text{(PBI-2)} :\quad
	e_{n+1}^{k} \leq \begin{cases}
	\delta (\gamma + \alpha)^k \text{ if } k \leq n, \\
	\displaystyle \delta \gamma^k \sum_{i=0}^{n}
	\binom{k}{i}\left(\frac{\alpha}{\gamma}\right)^i \text{ otherwise.}
	\end{cases}
}
\end{equation}
\paragraph{Error bound for PBI-3}
We use~\eqref{eq:seriesProduct} with $b_n=1$ for the series product to get
\begin{align}
a_n = \binom{n+k-1}{n}\beta^n
	= \frac{\prod_{l=1}^{k-1}(n+l)}{(k-1)!}\beta^n,
\end{align}
which yields the error bound
\begin{equation}
\boxed{\text{(PBI-3)} :\quad
	e^k_{n+1} \leq \delta
	\frac{\gamma^k}{(k-1)!} \sum_{i=0}^{n}\prod_{l=1}^{k-1}(i+l)\beta^{i}}
\end{equation}
for $k>0$.

\subsection{Full Primary Block Iteration}
We now consider a primary block iteration~\eqref{eq:primaryBlockIteration} with all block operators
non-zero,
\begin{equation}
\text{(PBI-Full)} :\quad \uvect^{k+1}_{n+1} =
\BMat^0_1\left(\uvect^{k}_{n+1}\right)
+ \BMat^1_0\left(\uvect^{k+1}_{n}\right) + \BMat^0_0\left(\uvect^{k}_{n}\right),
\label{eq:pbifull}
\end{equation}
with $\alpha$, $\beta$ and $\gamma$ defined in \eqref{alphabetagammadef}.
Applying Lemma~\ref{lem:primary} leads to
\begin{equation}
\rho_{k+1}(\zeta)
\leq \frac{\gamma + \alpha \zeta}{1-\beta\zeta} \rho_{k}(\zeta)
\quad \Longrightarrow\quad \rho_{k}(\zeta)
\leq \left(\frac{\gamma + \alpha\zeta}{1-\beta\zeta}\right)^{k}\rho_{0}(\zeta).
\end{equation}
Combining the calculations performed for PBI-2 and PBI-3, we obtain
\begin{align}
\rho_{k}(\zeta) &\leq \delta \zeta \gamma^k
\left[\sum_{n=0}^{k} \binom{k}{n}\left(\frac{\alpha}{\gamma}\right)^n\zeta^{n}\right]
\left[\sum_{n=0}^{\infty} \binom{n+k-1}{n}\beta^n\zeta^{n}\right]
\left[\sum_{n=0}^{\infty}\zeta^{n}\right] \\
&= \delta \zeta \gamma^k
\left[\sum_{n=0}^{k} \binom{k}{n}\left(\frac{\alpha}{\gamma}\right)^n\zeta^{n}\right]
\left[\sum_{n=0}^{\infty} \sum_{i=0}^{n} \binom{i+k-1}{i}\beta^i \zeta^{n}\right].
\end{align}
Then using \eqref{eq:seriesProduct} with
\begin{equation}
a_n = \begin{cases}
\displaystyle\binom{k}{n}\left(\frac{\alpha}{\gamma}\right)^n \text{ if } n
\leq k, \\
0 \text{ otherwise,}
\end{cases}\quad
b_n = \sum_{i=0}^{n} \binom{i+k-1}{i}\beta^i,
\end{equation}
we obtain
\begin{equation}
\rho_{k}(\zeta) \leq \delta \zeta \gamma^k \sum_{n=0}^{\infty} c_n \zeta^n,
\ \mbox{with}\
c_n = \sum_{i=0}^{\min(n, k)} \sum_{l=0}^{n-i}
\binom{k}{i}\binom{l+k-1}{l}
\left(\frac{\alpha}{\gamma}\right)^i\beta^l.
\end{equation}
From this we can identify the error bound
\begin{equation}
\boxed{\text{(PBI-Full)} :\quad
	e_{n+1}^k \leq \delta \gamma^k \sum_{i=0}^{\min(n, k)} \sum_{l=0}^{n-i}
\binom{k}{i}\binom{l+k-1}{l}
\left(\frac{\alpha}{\gamma}\right)^i\beta^l.}
\end{equation}
\bibliographystyle{siamplain}
\bibliography{pint,references}
\end{document}